\documentclass[12pt,a4paper]{amsart}
\usepackage[latin1]{inputenc}   % pour les accents
\usepackage{amsmath}
\usepackage{amsfonts}
\usepackage{amssymb}
\usepackage{esint}
\usepackage{graphicx}
\usepackage{amsthm}
\usepackage{enumerate}
\usepackage{verbatim}
\usepackage{hyperref}
\usepackage[english]{babel}
\usepackage{bbm}
\usepackage{mathrsfs}
\usepackage{relsize}
\usepackage{textcomp}
\usepackage{xcolor}

\def \R{\mathbb{R}}
\def \E{\mathbb{E}}
\def \P{\mathbb{P}}

\def\cP{\mathcal P}
\def\cK{\mathcal K}

\theoremstyle{plain} 
\newtheorem{thm}{Theorem}[section] 
\newtheorem{cor}[thm]{Corollary} 
\newtheorem{lem}[thm]{Lemma} 
\newtheorem{prop}[thm]{Proposition} 

\newtheorem{defn}[thm]{Definition}

\theoremstyle{definition} 
\newtheorem{rem}[thm]{Remark}
\numberwithin{equation}{section}
\newcommand{\be}{\begin{equation}}
	\newcommand{\ee}{\end{equation}}

\def\rife#1{(\ref{#1})}
\def\m{\noalign{\medskip}}

\def\dive{{\rm div}}
\def\de{\delta}

\def\vep{\varepsilon}
\def\elle#1{L^{#1}(\Omega)}

\newcommand{\miezz}{\frac{1}{2}}

\newcommand{\eps}{\varepsilon}

\newcommand{\into}{\ensuremath{\int_{\Omega}}}

\newcommand{\norm}[1]{\ensuremath{\left\Arrowvert #1 \right\Arrowvert}}
\newcommand{\norminf}[1]{\ensuremath{\left\Arrowvert #1 \right\Arrowvert_\infty}}

\newcommand{\weak}{\ensuremath{\rightharpoonup}}

\newcommand{\intG}{\int_{\Gamma_{\delta}^\eps}}
\newcommand{\tme}{\tilde m_\eps}
\newcommand{\bl}{\bar\lambda}

\begin{document}
	
	\title{Ergodic problems for second-order Mean Field Games with state constraints}
	
	\author{Alessio Porretta$^\dagger$}\thanks{$^\dagger$Dipartimento di Matematica, Universit\`a di Roma Tor Vergata. 
		Via della Ricerca Scientifica 1, 00133 Roma, Italy. \texttt{porretta@mat.uniroma2.it}} 
	
	\author{Michele Ricciardi$^\ddagger$}\thanks{$^\ddagger$King Abdullah University of Sciences and Technologies (KAUST), Saudi Arabia. \texttt{michele.ricciardi@kaust.edu.sa}}

	\date{\today}
	
	
	\maketitle
	
	\begin{abstract}
		We study an ergodic mean field game problem with state constraints. In our model the agents are affected by idiosyncratic noise and use a (singular) feedback control to prevent the Brownian motion from exiting the domain. We characterize the equilibrium as the (possibly unique) solution to a second-order MFG system, where the value function blows up at the boundary while the density of the players is smooth and flattens near the boundary as a consequence of the singularity of the drift induced by the feedback strategy of the agents.
	\end{abstract}

	\section{Introduction}

	Mean Field Game theory was introduced by J.-M. Lasry and P.-L. Lions in 2006 (\cite{LL1}, \cite{LL2}, \cite{LL-japan}) and extensively developed by  P.-L. Lions in his courses at Coll\`ege de France (2006-2012), while similar ideas related to Nash equilibria in large populations were introduced independently by P. Caines, M. Huang and R. Malham\'e \cite{HCM}.  In the last decade, the theory of Mean Field Games (hereafter, MFG) has hugely evolved in both theoretical aspects and applied models, and established as one of the key-tools for studying rational strategic interactions in multi-agent systems. We refer to \cite{CP-Cime}, \cite{CD}, for a broad presentation  of the topic from, respectively, the PDE and the probabilistic viewpoint. 
	
	In the typical framework of a mean-field game model, the agents are represented as dynamical states, which are controlled to maximize the individual utility function. The interaction occurs because the running costs  and the final pay-off of the single agent may depend on the collective behavior through the distribution law of the dynamical states. In a mean-field macroscopic approach, the equilibria are known to solve a system of PDEs (so-called MFG system) where the usual Hamilton-Jacobi-Bellman equation of the value function of the single agent is coupled with the Fokker-Planck (continuity) equation describing the density evolution of the population.
	
	It is very natural, in several applications, that the dynamical state of the representative agent needs to be confined in a given domain $\Omega$ (typically, a bounded subset of $\R^n$). This forces the players to restrict to a class of  controls that confine the dynamics in $\Omega$ (or at least in $\overline \Omega$). This is what is called a {\it state constraint control problem}, from the agents' viewpoint. It is well-known that the constrained optimization readily leads to formation of singularities for the value function as well as for the optimal controls, and the mathematical analysis of MFG systems in this setting becomes quite difficult. 
	This is the reason why there are only  few contributions  to the analysis of MFG with state constraints existing in the literature so far.
	
	In the case of deterministic dynamics, the analysis of MFG with state constraints has been exploited in \cite{CaCa}, \cite{CaCaCa}; the authors were able to use new refined semi-concavity results for the value function of deterministic control problems in order to develop a satisfactory PDE approach to the MFG system, which is a first-order system in that case. By contrast, it seems that the state constraint problem has not been explored so far for MFG of second-order, say with an underlying stochastic dynamics for the agents.
	The purpose of this note is to analyze a model problem of this kind, by considering the case that the individual dynamics are affected by an additive idiosyncratic noise (in the form of a standard Brownian motion). 
	
	To be more precise, we consider the case that the generic agent is represented by the simple SDE
	\begin{equation}\label{dyn}
		\begin{cases}
			dX_t=a(X_t)\,dt+\sqrt 2\, dB_t\,,\\
			X_0=x\in\Omega\,,
		\end{cases}
	\end{equation}
	where ${(B_t)}_t$ is a standard $n$-dimensional Brownian motion and the feedback control policy $a(\cdot)$ is used by the player to force $X_t$ to remain in $\Omega$ (almost surely). This is of course a singular optimal control problem, because the feedback control $a(X_t)$ has to blow up at the boundary to prevent the Brownian motion from exiting the domain. This kind of stochastic control problems were studied in a pioneering paper by J.-M. Lasry and P.-L. Lions \cite{LL-SC}, who were able to characterize the value function in terms of second-order equations with singular boundary conditions. In this paper, we are going to consider the associated MFG problem, namely the case where the cost functional of the agents depends itself (at equilibrium)  on the distribution law of the process. Because of the technical difficulties of state constraint problems, we consider here the stationary ergodic setting, which means that the optimization of the agents can rely on the stationary invariant measure of the underlying controlled dynamics.
	
	Let us now describe more precisely the mean field game problem that we consider. Let $m=m(x)$ be a probability density  (which can be assumed to be continuous, so far) in $\Omega$, and take a cost function $F= F(x,m)$ which can be assumed to be continuous in $\Omega\times \cP(\Omega)$ (the set of probability measures, endowed with the Wasserstein distance). For a given (exogenous) measure $m$, the optimal value of the individual agents (in the long run optimization)  is 
	$$
	\lambda:=  
	\lim\limits_{T\to \infty} \,\, \inf_{a\in {\mathcal A}}\E\frac1T \left[\int_0^{T} \left(F(X_t,m(X_t))+C_q|\alpha(X_t)|^{q}\right)\right]\,,
	$$
	where $X_t$ is given by \rife{dyn}, and $\mathcal{A}$ is a suitable set of feedback controls that constrain $X_t$ into $\Omega$; namely,
	$$
	\mathcal{A}:=\Big\{\alpha\in\mathcal{C}(\Omega;\R^n)\,\big\lvert\, \P\big(X_t\in\Omega\,\,\,\forall\, t>0\big)=1\Big\}\,.
	$$
	The running cost depending on the control is assumed here to be of power-type, and in the following we assume that $q\geq 2$. The   constant $C_q$ is just a normalization constant defined as $C_q=\frac1{(q-1)}\left(\frac q{q-1}\right)^{-q}$. Hereafter, we set $p= \frac q{q-1}$ the conjugate exponent of $q$, so that we are assuming
	$$
	1<p\leq 2\,.
	$$
	The value function of the ergodic control problem  satisfies
	$$
	u(x):= \inf_{a\in {\mathcal{A}}} \E  \left[\int_0^{\theta_a} \left(F(X_t,m(X_t))+C_q|\alpha(X_t)|^{q}\right)+ u(X_{\theta_a}) -\theta_a \, \lambda \right]
	$$
	where $\theta_a$ denotes any stopping time such that $\theta_a\leq T$.  For given fixed $F(x,m(x))$, it follows from \cite[Thm VII.3]{LL-SC} that $u$ is the unique solution (up to an addition of a constant) of  the elliptic problem
	\begin{equation}\label{LL-ergodic}
		\begin{cases}
			-\Delta u +|\nabla u |^p+\lambda  =F(x,m) & \hbox{in $\Omega$}\,, \\
			\m
			\lim\limits_{x\to\partial\Omega} u(x)=+\infty & 
		\end{cases}
	\end{equation}
	and  the (unique) optimal control is given in feedback form through
	\begin{equation}\label{eq:feedback}
		\alpha(x)=-p|\nabla u(x) |^{p-2}\nabla u(x) \,.
	\end{equation}
	Up to now, the measure $m$ is fixed and represents a rational anticipation made by the agents of the possible distribution law of the whole population. A Nash equilibrium (in the mean-field game approach) occurs if $m$ coincides with the stationary invariant measure (uniquely) associated with the optimal trajectory. In terms of PDEs, this means that 
	the equilibrium (a triple $(\bar \lambda, u,m)$) comes out from the solution of the following system 
	\begin{equation}\label{mfg}
		\begin{cases}
			-\Delta u + |\nabla u|^p+\bar\lambda=F(x,m)\,,\qquad x\in\Omega\,,\vspace{0.1cm}\\
			-\Delta m -p\,\mathrm{div}(m|\nabla u|^{p-2}\nabla u)=0\,,\qquad x\in\Omega\,,\vspace{0.1cm}\\
			\mathlarger{\into} m(x)\,dx=1\,,\qquad\lim\limits_{x\to\partial\Omega}u(x)=+\infty\,,
	\end{cases}\end{equation}
	In this note, we prove that the MFG system \rife{mfg} has a unique solution, and that  $u,m$ are smooth  inside $\Omega$ and $m$ is actually $C^1$ up to the boundary of $\Omega$. This is the main result of the paper.
	
	\begin{thm}\label{main}
		Let $\Omega$ be  a bounded domain in $\R^n$ with $\mathcal{C}^2$ boundary. Assume that $1<p\le 2$ and  that $F$ satisfies \rife{Fxm}.
		Then the problem \eqref{mfg} admits a solution $(\bl,u,m)$, with $u\in W^{2,r}_{loc}(\Omega)$ for all $r<+\infty$ and $m\in L^\infty(\Omega)\cap H^1(\Omega)$.
		In addition, we have:
		\vskip0.6em
		(i) $m\in C^{1, \alpha}(\overline \Omega)$ for some $\alpha>0$, and there exist constants $C_1,C_2$ such that 
		$$
		C_1\, d(x)^{p'}\leq m(x) \leq C_2 \, d(x)^{p'}
		$$
		for every $x\in \Omega$, where $d(x)=$dist$(x,\partial \Omega)$.
		\vskip0.4em
		(ii) $u\in L^1(dm)$ and the optimal feedback control satisfies 
		$$
		\alpha(x)= -p|\nabla u(x) |^{p-2}\nabla u(x)= -\frac{p'}{d(x)}\nu(x)+ O(1)\qquad \hbox{as $d(x)\to 0$.}
		$$
		
		In addition, if $F$ satisfies \rife{LL-mon},  $(\bl,u,m)$ is unique (up to an addition of a constant to $u$).
	\end{thm}

	The proof of Theorem \ref{main} relies on two main ingredients. One is the analysis of the singular stochastic control problem \rife{LL-ergodic}, which was developed in \cite{LL-SC}, and later in \cite{LP}, \cite{JuanSebastianVeron}. In particular, we borrow from those latter papers the refined asymptotics near the boundary for  $\nabla u(x)$, hence for the optimal feedback control \rife{eq:feedback}. The second ingredient is the analysis of the Fokker-Planck equation with singular drift. This is possible once we know the asymptotic behavior of $\nabla u(x)$ near the boundary, which allows us to use part of the theory developed in our former paper \cite{MFGinv} for the case of dynamics with invariance conditions. The case of invariance conditions means that  the process remains in $\Omega$  {\it for all controls available}; this is of course a much better situation for the controller, and in fact it results into  a more regular value function (proved to be Lipschitz and globally semi-concave, see \cite{MFGinv}). What we observe in the present article is that, in the second-order state constraint problem, the control and the value function become highly singular, but they ensure, after all, the invariance condition for the optimal dynamics. Therefore,  from the viewpoint of the population dynamics, this results into a regularization of the density function at equilibrium. The stronger the control will push, the flatter the density will be near the boundary, which is the spirit of estimates (i)--(ii) in Theorem \ref{main}. This is the main point that we address in our analysis and it seems peculiar of the diffusive character of the dynamics; to our knowledge, there is no similar behavior observed in the deterministic case. 
	
	Of course, the precise character of such boundary estimates is induced by the simple homogeneous form of the diffusion and of the cost functional, but we believe this could serve as a prototype for a more general analysis of second-order state constraint problems.  In addition, let us mention that the study of the ergodic problem \rife{mfg} will be certainly relevant for the analysis of the long-time behavior of exit-time problems, as is well-understood already for a single HJ equation \cite{BPT}.

	\section{Notations and assumptions}

	Throughout the paper, $\Omega$ denotes a  bounded open subset of $\R^n$, $n\geq 1$, and $\partial \Omega$ denotes its boundary. We assume that $\Omega$ is of class $C^2$ and we currently use 
	the so-called \emph{oriented distance function}, defined in this way:
	\begin{equation*}
		d_\Omega(x)=\left\{\begin{array}{rl}
			\mathrm{dist}(x,\partial\Omega)  & \quad\mbox{if }x\in\Omega\,,\\
			-\mathrm{dist}(x,\partial\Omega) & \quad\mbox{if }x\notin\Omega\,.
		\end{array}\right.
	\end{equation*}
	We will simply write $d(\cdot)$ when we will refer to the set $\Omega$. In other cases, we will specify the involved set. It is well-known (see e.g. \cite{GT}) that $d(x)$ is a $C^2$ function in a suitable neighborhood of $\partial\Omega$. For $\de>0$,  we define   the open sets 
	$$
	\Omega_\delta:=\{x\in\Omega\mbox{ s.t. }d(x)>\delta\}\,;\,\qquad \Gamma_\de:= \Omega\setminus \overline{\Omega_\de}=\{x\in \Omega\,:\,d(x)<\de\}\,.
	$$
	The outward unit normal vector on $\partial\Omega$ is denoted by $\nu(x)$. Since $\Omega$ is smooth, we also have $\nu(x)=-\nabla d(x)$ for $x\in \Gamma_\de$, provided $\de$ is sufficiently small. The tangent vector on $\partial \Omega$ is denoted by $\tau(x)$. The notation $v\otimes w$ is used for 
	the tensor product of vectors: if $v\in\R^n$ and $w\in\R^m$, then $v\otimes w$ is the $n\times m$ symmetric matrix such that $(v\otimes w)_{ij}=v_i w_j$.
	
	For $x_0\in\R^n$ and $r>0$, the ball centered in $x_0$ of radius $r$ will be indicated by $B(x_0;r)$. If $x_0$ is the origin $O$, we simply write $B_r$.
	
	As it is usual, we denote by $W^{1,\infty}(\Omega)$ the space of bounded, Lipschitz functions in $\Omega$, and by $\cP(\Omega)$ the space of probability measures in $\Omega$, endowed with the standard Kantorovich-Rubinstein distance. We use the  standard notation $L^p(\Omega)$, $H^1(\Omega)$ for Lebesgue spaces and, respectively, for the  Sobolev space where $u,\nabla u$ are $L^2$-summable. The notation $W^{2,p}(\Omega)$ is used for the Sobolev space where second derivatives are in $L^p$, and $C^{k,\alpha}(\Omega)$ denotes the space where $j$-derivatives are $\alpha$-H\"older continuous with $j\leq k$. We shortly write $C^\alpha(\Omega)$ for $C^{0,\alpha}(\Omega)$.
	Finally, we write  $ f(x)= O(g(x))$ (as $x\to x_0$, or $x\to \partial\Omega$) to say that, for some constant $C>0$, we have $|f(x)|\leq C g(x)$ (for $x$ near $x_0$, or $x$ near $\partial\Omega$); as it is standard, we write instead $ f(x)= o(g(x))$ to mean that $\frac{f(x)}{g(x)}\to 0$ and $ f(x) \sim g(x)$ whenever  $\frac{f(x)}{g(x)}\to 1$.
	\vskip1em
	The goal of the paper is to study the MFG system
	$$
	\begin{cases}
		-\Delta u + |\nabla u|^p+\bar\lambda=F(x,m)\,,\qquad x\in\Omega\,,\vspace{0.1cm}\\
		-\Delta m -p\,\mathrm{div}(m|\nabla u|^{p-2}\nabla u)=0\,,\qquad x\in\Omega\,,\vspace{0.1cm}\\
		\mathlarger{\into} m(x)\,dx=1\,,\qquad\lim\limits_{x\to\partial\Omega}u(x)=+\infty\,,
	\end{cases}
	$$
	where $1<p\leq 2$.  As for the coupling cost function $F(x,m)$, we will assume that $F:\Omega \times \cP(\Omega) \to \R$ is a continuous function such that  $x\mapsto F(x,m)$ is Lipschitz continuous, for every $m\in \cP(\Omega)$ and satisfies
	\be\label{Fxm}
	\exists \, \, C_F>0\,:\, \quad \|F(\cdot, m)\|_{W^{1,\infty}(\Omega)} \leq C_F \quad \forall m\in \cP(\Omega)\,.
	\ee
	While assumption \rife{Fxm} will be used for the existence of a solution, the uniqueness will be proved under the additional monotonicity condition:
	\be\label{LL-mon}
	\into \left( F(x,m)-F(x,\tilde m)\right)\, d(m-\tilde m)(x) \geq 0 \qquad \forall m\,,\,\tilde m \in \cP(\Omega)\,.
	\ee

	\section{The Hamilton-Jacobi-Bellman equation}
	
	In this Section, we consider the single stochastic control problem of the agents, and we start by recalling what is known about the ergodic problem
	\begin{equation}\label{eq:u}
		\begin{cases}
			-\Delta u+|\nabla u|^p+\bl=f(x)\,,\\
			\lim\limits_{x\to\partial\Omega} u(x)=+\infty\,.
		\end{cases}
	\end{equation}
	for a given source $f(x)$. 
	
	Problem \rife{eq:u} was obtained in \cite{LL-SC} as the limit of  discounted control problems,
	% corresponding to the equation 
	and the existence of  a unique solution $(\bar \lambda , u)$ was proved (up to an addition of a constant to $u$) assuming $f\in \elle\infty$ (or only $f\in L^\infty_{loc}(\Omega)$ such that $f=O(d(x)^{-\frac p{p-1}})$ as $d(x)\to 0$).  In particular, $u\in W^{2,r}_{loc}(\Omega)$ for every $r<\infty$ and satisfies the following asymptotic estimates:
	%Some asymptotic estimates for the function $u$ and its derivative were also obtained in \cite{LL-SC,LP}:
	%\begin{itemize}
	%\item If $F\in L^\infty(\Omega\times\mathcal{P}(\Omega))$, then
	\begin{equation}\label{eq:stimeu}
		\begin{split}
			&1<p<2 \quad\implies  u(x)\sim \beta_p d(x)^{2-p'}\quad\mbox{ for }x\to\partial\Omega,\mbox{ with } \beta_p=\frac{(p-1)^{1-p'}}{p'-2}\,;\\
			&p=2 \qquad\hspace{0.35cm}\implies u(x)\sim -\log d(x)\quad\hspace{0.15cm} \mbox{ for }x\to\partial\Omega\,; 
		\end{split}
	\end{equation}
	%\begin{comment}
	%\item More precisely, we have the following expansion at the boundary (see \emph{Theorem II.3} of \cite{LL-SC}):
	%\begin{align}\label{eq:tayloru}
	%\begin{array}{ll}\tag{1.6b}
	%\displaystyle1<p<\frac32 &\implies u(x)=c_p d(x)^{-\frac{2-p}{p-1}}\big(1+O(d)\big);\\
	%\displaystyle p=\frac 32 &\implies u(x)=4d(x)^{-1}\big(1+O(d|\log d|)\big);\\
	%\displaystyle \frac32<p<2 &\implies u(x)=c_pd(x)^{-\frac{2-p}{p-1}}+O(1)\,;\\
	%\displaystyle p=2 &\implies u(x)=-\log d(x)+O(1)\,;
	%\end{array}
	%\end{align}
	%\end{comment}
	Recall that $2-p'=\frac{p-2}{p-1}<0$ if $1<p<2$, hence $u$ blows up for $x\to\partial\Omega$, as required in the boundary condition. 
	%Moreover, for $p=2$ we have $1-p'=-1$. Hence, in all cases the gradient estimate can be formally deduced by differentiation of the estimate on $u$.
	For simplicity, we define $c_p=(p-1)^{1-p'}$ for $1<p\le2$ and the function 
	\begin{equation}\label{def:phi}
		\phi_p(x)=\left\{
		\begin{array}{rll}
			\frac1{p'-2}x^{2-p'} && \mbox{for }1<p<2\,,\\
			-\log x && \mbox{for }p=2\,.
		\end{array}
		\right.
	\end{equation}
	Then 
	%$\phi_p'(x)=-x^{1-p'}$ for all $1<p\le2$. Moreover, 
	the estimate \eqref{eq:stimeu} can be shortened to
	$$
	u(x)\sim c_p\,\phi_p(d(x))\qquad\mbox{for }x\to\partial\Omega\,,\quad\mbox{for }1<p\le2\,.
	$$
	In fact, if one observes that, after \rife{eq:stimeu}, it holds  $u(x)=o(d(x)^{-\frac1{p-1}})$, then it is possible to have a second-order expansion of $u$ for the solution of the ergodic problem. This is a  consequence of   \cite[Thm II.3]{LL-SC} and yields the following result.
	
	\begin{prop} Let $1<p\le2$ and $f\in L^\infty_{loc}(\Omega)$ such that $f(x)= o(d(x)^{-\frac1{p-1}})$. Then the unique solution $u$ of 
		\rife{eq:u} satisfies
		\begin{equation}\label{eq:taylorDu}
			\displaystyle u(x)=c_p\, \phi_p(d(x))+O\Big(\gamma_p\big(d(x)\big)\Big)\,,
		\end{equation}
		where $\phi_p$ is given by \eqref{def:phi} and $\gamma_p(\delta)$ is defined as follows:
		\begin{equation*}
			\gamma_p(\delta)=\left\{
			\begin{array}{llll}
				\displaystyle\delta^{3-p'} &\quad \mbox{if }&1<p<\frac32 &\quad(\mbox{i.e. }p'>3)\,,\\
				\displaystyle|\log\delta| &\quad \mbox{if }&p=\frac32 &\quad(\mbox{i.e. }p'=3)\,,\\
				\displaystyle 1 &\quad \mbox{if }&\frac32<p\le2 &\quad(\mbox{i.e. }2\le p'<3)\,.\\
			\end{array}
			\right.
		\end{equation*}
		\qed
	\end{prop}
	Assuming more regularity on the right-hand side $f$, similar asymptotics were obtained for $\nabla u$ (hence for the optimal feedback control) in \cite{LP}, \cite{JuanSebastianVeron}.  In particular, if $f\in W^{1,\infty}(\Omega)$ 
	then we have, for every $1<p\le2\,$:
	\begin{equation}\label{eq:stimeDu}
		\nabla u(x)=\big((p-1)d(x)\big)^{1-p'}\big[\nu(x)+O(d(x))\big]\,,\qquad \hbox{as $d(x)\to 0$,}
	\end{equation}
	and in particular there exists $\psi\in L^\infty(\Omega)$ such that 
	\be\label{stime-feed}
	\alpha(x)=-\frac{p'}{d(x)}\nu(x)+\psi(x)
	\ee
	where $\alpha(\cdot)$ is the optimal feedback law \rife{eq:feedback}.
	
	Notice that, in all cases, the gradient estimate can be formally deduced by differentiation of the estimate on $u$. In fact, since $\phi_p'(x)=-x^{1-p'}$,  we have shortly
	$\nabla u(x)\sim c_p\,\phi_p'(d(x))\nabla d(x)$ as $d(x)\to 0$.
	%
	%  If $F(\cdot,m)\in W^{1,\infty}$ uniformly in $m$, then we have an asymptotic estimate also for the optimal feedback control $\alpha$ and the normal derivative $\partial_\nu u$: for $x\to\partial\Omega$,
	%\begin{equation}\label{eq:stimeDu}
	%\begin{split}
	%\alpha(x)=-\frac{p'}{d(x)}\nu(x)+\psi(x)\,,\qquad\quad
	%\nabla u(x)=\big((p-1)d(x)\big)^{1-p'}\big[\nu+O(d(x))\big]\,,
	%\end{split}
	%\end{equation}
	%where $\nu$ is the outward normal at $\partial\Omega$, $\psi\in L^\infty(\Omega)$ and the estimates hold for all $1<p\le2\,$;
	\begin{comment}\item More precisely, we have the following expansion at the boundary (see \emph{Theorem 1.2} of \cite{LP}):
		\begin{align}
			\begin{array}{ll}\tag{1.7b}
				\displaystyle1<p<\frac32 &\implies \nabla u(x)=c_p d(x)^{-\frac{2-p}{p-1}}\big(1+O(d)\big);\\
				\displaystyle p=\frac 32 &\implies u(x)=4d(x)^{-1}\big(1+O(d|\log d|)\big);\\
				\displaystyle \frac32<p<2 &\implies u(x)=c_pd(x)^{-\frac{2-p}{p-1}}+O(1)\,;\\
				\displaystyle p=2 &\implies u(x)=-\log d(x)+O(1)\,;
			\end{array}
		\end{align}
	\end{comment}

	%In this section we want to improve the regularity of $u$, showing a result involving the behavior of the Hessian $D^2u$ near the boundary $\partial\Omega$.
	%
	%Furthermore, we need to prove a well-posedness result for a different Hamilton-Jacobi-Bellman problem. This result is an adaptation of the parabolic case proved in \cite{MFGinv} and plays a crucial role to show the well-posedness of the Fokker-Planck equation.
	
	%Throughout the paper, we define for $\delta>0$ the set $\Omega_\delta:=\{x\in\Omega\mbox{ s.t. }d(x)>\delta\}$.
	
	\subsection{Further regularity for $u$} Now we wish to improve the regularity of $u$, by showing  the asymptotic behavior of the Hessian $D^2u$ near the boundary $\partial\Omega$.  We already have the estimates \eqref{eq:taylorDu} and \eqref{eq:stimeDu} for $u$ and $\nabla u$. 
	Another equivalent formulation of \eqref{eq:taylorDu} is the following one:
	\begin{equation}\label{eq:taylorDu2}
		\displaystyle u(x)=c_p\, \phi_p(d(x))\Big[1+O\Big(\omega_p\big(d(x)\big)\Big)\Big]\,,
	\end{equation}
	where $\omega_p(\delta)=\phi_p(\delta)^{-1}\gamma_p(\delta)$, i.e.:
	\begin{equation}\label{def:omega}
		\omega_p(\delta)=\left\{
		\begin{array}{llll}
			\displaystyle\delta &\quad \mbox{if }&1<p<\frac32 &\quad(\mbox{i.e. }p'>3)\,,\\
			\displaystyle \delta|\log\delta| &\quad \mbox{if }&p=\frac32 &\quad(\mbox{i.e. }p'=3)\,,\\
			\displaystyle \delta^{p'-2} &\quad \mbox{if }&\frac32<p<2 &\quad(\mbox{i.e. }2< p'<3)\,,\\
			\displaystyle |\log\delta|^{-1} &\quad \mbox{if }&p=2 &\quad(\mbox{i.e. }p'=2)\,.\\
		\end{array}
		\right.
	\end{equation}
	
	Since the function $\omega_p$ plays a role in other estimations, the formulation \eqref{eq:taylorDu2} is often used in the rest of the paper. One of them is the next result, which involves the second derivatives of $u$.
	
	\begin{prop}\label{prop:21}
		Assume that $1<p\le2$ and let $f\in\mathcal{C}^\alpha(\Omega)$, for $0<\alpha<1$, and $u$ be a solution of \eqref{eq:u}. Then we have
		\begin{equation}\label{eq:stimeD2u}
			\lim\limits_{x\to\partial\Omega}d(x)^{p'}D^2u(x)=(p-1)^{-p'}\nu(x)\otimes\nu(x)\,.
		\end{equation}
		Moreover, for $p\neq2$ we have a second-order expansion of $D^2u$:
		\begin{equation}\label{eq:taylorD2u}
			D^2u(x)=(p-1)^{-p'}d(x)^{-p'}\Big[\nu(x)\otimes\nu(x)+O\Big(\omega_p\big(d(x)\big)\Big)\Big]\,,
		\end{equation}
		where $\omega_p$ is defined in \eqref{def:omega}.
		%, and where we use the tensor product notation for vectors: if $v\in\R^n$ and $w\in\R^m$, then $v\otimes w$ is the $n\times m$ symmetric matrix such that $(v\otimes w)_{ij}=v_i w_j$.
	\end{prop} 
	In the following, we will often write $d$ and $\omega_d$ instead of $d(x)$ and $\omega_p(d(x))$ when there is no possibility of mistakes, in particular with the notation $\phi_p(d)$ and $O(\omega_d)$.
	\begin{proof}
		\emph{Step 1.} In the first step, we prove \eqref{eq:stimeD2u}. 
		We follow the approach used in \cite[Theorem 2.3]{JuanSebastianVeron} for the asymptotics of $\nabla u$.  Let $d_0$ be sufficiently small so that $d(x)$ is smooth in $\Gamma_{d_0}$.  We fix $x\in\partial\Omega$ and we consider a new orthonormal system of coordinates $(y_1,\dots,y_n)$ centered at $x$, where $y_1=-\nu(x)$. In these coordinates, we define for $0<\de< d_0 $, $0<\sigma<\miezz$ and $O_{d_0}:=(d_0,0,\dots,0)$,
		\begin{equation}\label{eq:od}
			D_\delta=B_{\delta^{1-\sigma}}\cap  B(O_{d_0};d_0)\,.
		\end{equation}
		Since $\sigma>0$, we note that $\frac 1{\delta} D_\delta\overset{\delta\to 0^+}\longrightarrow\R^n_{+}=\{\xi\in\R^n\,\arrowvert\, \xi_1>0\}\,$. Now we make another change of variable and define the following quantities:
		\begin{equation}\label{eq:xi}
			\xi=\frac{y}{\delta}\,,\qquad v_\delta(\xi):=\left\{\begin{array}{lll}
				\displaystyle c_p^{-1}\delta^{p'-2}u(\delta\xi) &\quad \mbox{if }&p\neq 2\,,\vspace{0.2cm}\\
				\displaystyle u(\delta\xi)+\log\delta &\quad \mbox{if }&p=2\,,\\
			\end{array}\right.
		\end{equation}
		where $c_p=(p-1)^{1-p'}$.  Observe that, since $\sigma<\miezz$, $d(\de\xi)\simeq \de\,\xi_1$ for $\de$ small; therefore, 
		thanks to  \eqref{eq:taylorDu}, we know that $v_\delta$ is locally bounded for $\xi_1>0$, uniformly with respect to $\delta$. Moreover, $v_\delta$ satisfies the following Hamilton-Jacobi equation:
		\begin{equation}\label{eq:hjbvd}
			-\Delta v_\delta(\xi) + c_p^{p-1}\,|\nabla v_\delta(\xi)|^p+c_p^{-1}\bl\,\delta^{p'}=c_p^{-1}\,\delta^{p'}\,f(\delta\xi)\,,\qquad\xi\in\frac 1\delta D_\delta\,,
		\end{equation}
		where we notice, due to \rife{eq:stimeDu}, that $|\nabla v_\de|$ is also locally bounded, for $\xi_1>0$. 
		By well-known elliptic regularity  (see \cite{GT}),  we infer that $v_\delta$ is locally bounded in $\mathcal{C}^{2+\alpha}$, uniformly in $\delta$. Then, with a bootstrap  argument, we easily obtain the relative compactness of $\{v_\delta\}_\delta$ in $\mathcal{C}^2_{loc}$. Then there exists a subsequence $\{\delta_k\}_k$ and a function $v$ such that $v_{\delta_k}\to v$. Passing to the limit in \eqref{eq:hjbvd} we get
		\begin{equation}\label{eq:v}
			-\Delta v(\xi) + c_p^{p-1}\,|\nabla v(\xi)|^p=0\,.
		\end{equation}
		To find an explicit representation of $v$, we consider two cases. For $p\neq2$ we come back to \eqref{eq:stimeu}. Using our change of variables we find
		$$
		\lim\limits_{x\to\partial\Omega} (p'-2)d(x)^{p'-2}u(x)=c_p\implies\lim\limits_{y_1\to0} y_1^{p'-2}u(y)=c_p(p'-2)^{-1}\,.
		$$
		Since $y=\xi\delta$, we have $y_1 =\delta\xi_1$, so  $y_1\to 0 \iff \delta\to 0$. Hence  the above rewrites as 
		$$
		\lim\limits_{\delta\to0}\, (\delta\xi_1)^{p'-2}u(\delta\xi)=c_p(p'-2)^{-1}
		$$
		which yields
		$$
		\lim\limits_{\delta\to0} v_\delta(\xi)=\frac1{p'-2}\,\xi_1^{2-p'}\,.
		$$
		For $p=2$ we know from \eqref{eq:taylorDu} that  $u(\cdot)+\log d(\cdot)\in L^\infty(\Omega)$; therefore, from \rife{eq:v} we deduce  that the function $w=e^{-v}$ is positive and harmonic in the half-space $\R^n_{+}$, with $w\leq c\, \xi_1$ for some $c>0$. This implies that there exists $\lambda >0$ such that $w= \lambda \xi_1$, which means that 
		$$
		\lim\limits_{\delta\to0} v_\delta(\xi)=-\log\xi_1 + \lambda
		$$
		We can summarize the two limits as
		$$
		\lim\limits_{\delta\to0}v_\delta(\xi)=\phi_p(\xi_1)\qquad\forall\,1<p\le2\,,
		$$
		and this limit is true for the whole sequence $\{v_\delta\}_\delta$ and for $\xi\in\R^n_+$.  Since the convergence holds in $\mathcal{C}^2_{loc}$, we also have
		$$
		\lim\limits_{\delta\to0} \nabla^2v_\delta(\xi)=\nabla^2\big(\phi_p(\xi_1)\big)\,.
		$$
		We compute the two Hessian matrices. For the right-hand side, we easily have
		$$
		\nabla^2\big(\phi_p(\xi_1)\big)=(p'-1)\,\xi_1^{-p'}\begin{pmatrix}
			1 & 0 &\dots & 0\\
			0 & 0 &\dots & 0\\
			\dots&\dots&\dots&\dots\\
			0 &0 &\dots &0
		\end{pmatrix}\,.
		$$
		As regards the left-hand side, we have
		$$
		\nabla^2 v_\delta(\xi)=c_p^{-1}\, \delta^{p'}\,\nabla^2 u(\delta\xi)\,.
		$$
		This implies
		$$
		\lim\limits_{y_1\to0}\partial^2_{y_iy_j} u(y)=0\qquad\mbox{for }(i,j)\neq (1,1)\,,
		$$
		and, choosing $\xi=(1,0,\dots,0)$, we get
		$$
		\lim\limits_{\delta\to 0}\,\delta^{p'}\,\partial^2_{y_1y_1} u(\delta,0,\dots,0)=c_p(p'-1)=(p-1)^{-p'}\,.
		$$ 
		Since $\partial^2_{y_1y_1}u(\delta,0,\dots,0)=\partial^2_{\nu(x)\nu(x)}u(x-\delta\nu(x))$ we have proved \eqref{eq:stimeD2u}.\\
		
		\emph{Step 2.} Now we  improve the equation \eqref{eq:stimeD2u} with a ``second-order expansion" near $\partial\Omega$, i.e. proving \eqref{eq:taylorD2u}.
		
		Subtracting \eqref{eq:v} from \eqref{eq:hjbvd}, we obtain that the function $v_\delta-v$ satisfies
		\begin{equation}\label{eq:vd-v}
			-\Delta(v_\delta-v)+V_\delta(\xi)\cdot\nabla(v_\delta-v)=g_\delta(\xi)\,,\qquad \xi\in\frac{1}{\delta}D_\delta\,,
		\end{equation}
		where
		\begin{equation*}
			\begin{split}
				&V_\delta(\xi)=p\int_0^1\big|t\nabla v_\delta(\xi)+(1-t)\nabla v(\xi)\big|^{p-2}\big(t\nabla v_\delta(\xi)+(1-t)\nabla v(\xi)\big)\,dt\,,\\
				&g_\delta(\xi)=-c_p^{-1}\bl\,\delta^{p'}+c_p^{-1}\delta^{p'}f(\delta\xi)\,.
			\end{split}
		\end{equation*}
		Since $V_\delta$ is locally bounded in $C^\alpha\left(\frac1\delta D_\delta\right)$ uniformly in $\delta$, we can use the local estimates of  \cite[Thm 6.2]{GT}. Called $e_1=(1,0,\dots,0)$, we have 
		\begin{equation*}
			\begin{split}
				\big|D^2v_\delta(e_1)-D^2v(e_1) \big|\le C\big(\norm{g_\delta}_{C^\alpha}+\norminf{v_\delta-v}\big)\,.
			\end{split}
		\end{equation*}
		As regards $g_\delta$, we have
		$$
		\norm{g_\delta}_{C^\alpha}\le C\delta^{p'}\norm{f}_{C^\alpha}+C\delta^{p'}\le C\delta^2\,.
		$$
		Now we analyze the term $\norminf{v_\delta-v}$. We restrict here to $p\ne 2$. We rewrite   as
		\begin{equation}\label{eq:vdmv}
			\begin{split}
				v_\delta(\xi)-v(\xi)=c_p^{-1}\delta^{p'-2}\Big[u\big(\delta\xi_1,\dots,\delta \xi_n\big)-c_p\phi_p\big(\delta\xi_1\big)\Big]\,.
			\end{split}
		\end{equation}
		Using the estimates in \eqref{eq:taylorDu2}, we know that
		$$
		u(\delta\xi_1,\dots,\delta \xi_n)=c_p\phi_p(\delta\xi_1)\big(1+O(\omega_\delta)\big)\,.
		$$
		Plugging this estimate into \eqref{eq:vdmv} we find $\norm{v_\delta-v}=O(\omega_\delta)$, which implies
		$$
		\big|D^2v_\delta(e_1)-D^2v(e_1)\big|\le C\delta^2+O(\omega_\delta)=O(\omega_\delta).
		$$
		We compute $D^2 v_\delta(e_1)$ and $D^2v(e_1)$ as in the previous step. This implies
		$$
		\langle D^2u(x)\nu(x),\nu(x)\rangle=(p-1)^{-p'}d(x)^{-p'}\big(1+O(\omega_d)\big)\,.
		$$
		Since $\langle D^2u(x)\tau(x),\tau(x)\rangle$, $\langle D^2u(x)\tau(x),\nu(x)\rangle\in o(\langle D^2u(x)\nu(x),\nu(x)\rangle)$, we have proved \eqref{eq:taylorD2u}.
	\end{proof}
	
	\subsection{The linearized HJB equation.} 
	
	As it is often needed in the study of MFG systems, we will need to analyze the properties of the linearized HJB equation. In particular, we need 
	a well-posedness result for a transport-diffusion problem with singular drift. This result is an adaptation of the parabolic case proved in \cite{MFGinv} and plays a crucial role to show the well-posedness of the Fokker-Planck equation.
	\vskip0.4em
	
	Let $\delta\ge0$, $g\in L^\infty(\Omega)$, $b\in L^\infty_{loc}(\Omega)$. We study the following   equation
	\begin{equation}\label{hjb2}
		\begin{cases}
			-\Delta\phi + b(x)\cdot\nabla\phi+\delta\phi=g(x)\,,\\
			\phi\in L^\infty(\Omega)\,,
		\end{cases}
	\end{equation}
	where an \emph{invariance condition} is prescribed on the underlying dynamics, as in \cite{MFGinv}. In the present setting, the invariance  condition  reads as the following assumption on the drift term: $\exists\delta_0, C>0$ s.t.
	\begin{equation}\label{eq:inv}
		\Delta d(x)-b(x)\cdot\nabla d(x)\ge\frac1{d(x)}-Cd(x)\,,\qquad\forall\,x\in\Gamma_{\delta_0}\,.
	\end{equation}
	We note that this condition is satisfied for the linearized HJB equation arising from \rife{eq:u}.  Indeed, if  $b(x)=p|\nabla u|^{p-2}\nabla u=-\alpha(x)$, by \eqref{stime-feed} we have
	\begin{equation*}
		\begin{split}
			-b(x)\cdot\nabla d(x)=\frac{p'}{d(x)}(1+o(1))\,,
		\end{split}
	\end{equation*}
	which clearly implies \eqref{eq:inv} since $p'>1$.  In fact, we note that  in this case even a  stronger condition than \rife{eq:inv} is satisfied, namely
	\begin{equation}\label{inv_stronger}
		\begin{split}
			\exists \,\, \beta_0>1 \,:\,\quad \Delta d(x)-b(x)\cdot\nabla d(x)\ge\frac{\beta_0}{d(x)}(1+o(1))\,,\qquad\forall\,x\in\Gamma_{\delta_0}\,.
		\end{split}
	\end{equation}
	%with $\beta_0>1$. This clearly implies \eqref{eq:inv}.
	
	%The definition of a solution is given in a weak sense.
	\begin{defn}\label{defhjb2}
		We say that $\phi\in L^\infty(\Omega)\cap H^1_{loc}(\Omega)$ is a weak solution of \eqref{hjb2} if for all $\xi\in C^\infty_c(\Omega)$ it holds
		$$
		\into \big(\nabla\phi(x)\cdot(\nabla\xi(x)+b(x)\xi(x))+\delta\phi(x)\xi(x)\big)\,dx=\into g(x)\xi(x)\,dx
		$$
	\end{defn}
	
	We start with the case $\delta>0$.
	\begin{prop}\label{prop:hjdelta}
		Let $\delta>0$, $g\in L^\infty(\Omega)$ and assume $b\in L^\infty_{loc}(\Omega)$ satisfies \eqref{eq:inv}. Then there exists a unique solution of \eqref{hjb2}, in the sense of Definition \ref{defhjb2}.
	\end{prop}
	\begin{proof}
		For the existence part, we consider the set of functions $\{\phi_\eps\}_\eps\subseteq L^\infty(\Omega)$ which 
		%such that $\phi_\eps=0$ in $\Omega\setminus\Omega_\eps$ and 
		solves in $\Omega_\eps$ the equation
		\begin{equation}\label{eq:lella}
			\begin{cases}
				-\Delta\phi_\eps(x)+b(x)\cdot\nabla\phi_\eps(x)+\delta\phi_\eps(x)=g(x)\,,\\
				\nabla\phi_\eps(x)\cdot\nu(x)_{|\partial\Omega_\eps}=0\,.
			\end{cases}
		\end{equation}
		Since $g\in L^\infty(\Omega)$ and $\delta>0$, we have by the maximum principle
		$$
		\sup\limits_\eps\norminf{\phi_\eps}\le C\implies\exists\phi\in L^\infty(\Omega)\,\mbox{ s.t. }\phi_\eps\weak^*\phi\,,
		$$
		with a weak*-convergence in $L^\infty(\Omega)$ (here we consider $\phi_\eps$ as the zero extension in $\Omega\setminus\Omega_\eps$).  We also have, in a standard way, local energy estimates. Indeed,  multiplying the equation of $\phi_\eps$ for $\phi_\eps\xi_r^2$, where $\xi_r\in C_c^\infty(\Omega)$ s.t. $\xi_r=1$ in $\Omega_r$ and $\norminf{\nabla\xi_r}\le\frac 2{r}$, we find
		$$
		\into |\nabla\phi_\eps|\xi_r^2\,dx+2\into\nabla\phi_\eps\cdot \nabla\xi_r\,\phi_\eps\xi_r\,dx+\into b\cdot\nabla\phi_\eps\,\phi_\eps\xi_r^2\,dx=\into (g-\delta\phi_\eps)\phi_\eps\xi_r^2\,dx\,.
		$$
		Using a generalized Young's inequality and the  local $L^\infty $ bound on $b$, we get
		$$
		\int_{\Omega_r} |\nabla\phi_\eps|^2\,dx\le C_r\,,
		$$
		for a certain $C_r>0$ not depending on $\eps$. Hence $\phi_\eps$ is bounded  in $L^\infty(\Omega)$ and locally in $H^1$, uniformly in $\eps$, which is enough to pass to the limit for $\eps\to 0$ in the weak formulation of $\phi_\eps$, with test function $\xi\in C_c^\infty(\Omega)$. This concludes the existence part.
		
		As regards uniqueness, we use the same argument  as exploited in similar contexts in \cite[Thm 6.2]{BCR}, \cite[Thm 3.8]{MFGinv}. We consider two bounded solutions $\phi$ and $\varphi$; for $M,\eps>0$, we consider a slight perturbation of $\phi$, namely
		$$
		\phi_\eps=\phi+\eps(M-\log d(x))\,.
		$$
		We obtain
		\begin{align*}
			-\Delta\phi_\eps+b\cdot\nabla\phi_\eps+\delta\phi_\eps=g+\eps\left(\frac{\Delta d}{d}-\frac{b\cdot\nabla d}d-\frac{|\nabla d|^2}{d^2} + \de(M-\log d)\right)\,.
		\end{align*}
		The term in the parenthesis is non-negative for $M$ large, thanks to \eqref{eq:inv}. Moreover, $\phi_\eps$ blows up at the boundary, whereas $\varphi$ is a bounded function. Hence, we get
		$$
		-\Delta\phi_\eps+b\cdot\nabla\phi_\eps+\delta\phi_\eps\ge-\Delta\varphi+b\cdot\varphi+\delta\varphi\,,\qquad {\phi_\eps}_{|\partial\Omega} >\varphi_{|\partial\Omega}\,,
		$$
		which implies, for the maximum principle, $\phi_\eps\ge\varphi$ $\forall\eps>0$. Hence $\phi\ge\varphi$. Exchanging the role of $\phi$ and $\varphi$, we get $\phi=\varphi\,$.
	\end{proof}
	
	With a strengthening of hypotheses, we can also prove a Lipschitz bound for the solutions of \eqref{hjb2}. This result holds for all $\delta\ge0$ and is a crucial estimate for the case $\delta=0$. The strategy goes through uniform gradient bounds  for the approximated problems.
	
	{\color{black}
		
		\begin{prop}\label{prop:Lip}
			Let $\eps>0$. We consider $\{\delta_\eps\}_\eps\subset\R_{\ge0}$, $\{g_\eps\}_\eps\subset W^{1,\infty}(\Omega)$, which are uniformly bounded, respectively, in $\R$ and $W^{1,\infty}(\Omega)$. Assume that  $\{b_\eps\}_\eps$ is bounded  in $W^{1,\infty}(\cK)$ for any compact subset $\cK\subset\Omega$ and satisfies, for a certain $C\geq0$ and $\delta_0,\gamma_0>0$, \eqref{eq:inv} and
			\begin{equation}\label{eq:condJac}
				\mathrm{Jac}\,b_\eps(x)\ge -Cd(x)^{\gamma_0-2}I_{n\times n}\,,\qquad \forall\,x\in\Gamma_{\delta_0}\,,\forall\,\eps>0\,,
			\end{equation}
			where $I_{d\times d}$ is the identity matrix in $\R^{d\times d}$. Then the solution $\phi_\eps$ of the following problem
			\begin{equation}\label{eq:phieps}
				\begin{cases}
					-\Delta\phi_\eps(x)+b_\eps(x)\cdot\nabla\phi_\eps(x)+\delta_\eps\phi_\eps(x)=g_\eps(x)\,,\quad x\in \Omega_\eps\\
					\nabla\phi_\eps(x)\cdot\nu_\eps(x)_{|\partial\Omega_\eps}=0
				\end{cases}
			\end{equation}
			satisfies
			\begin{equation}\label{gradest}
				\begin{split}
					\|\nabla\phi_\eps\|_{L^\infty(\Omega_\eps)} \leq K \|  g_\eps\|_{W^{1,\infty}(\Omega_\eps)}
				\end{split}
			\end{equation}
			for some $K$ depending on $\Omega$, $\gamma_0$ and $b$ (through \rife{eq:inv} and \rife{eq:condJac}), but independent on $\vep$.
			
			Besides, suppose that, instead of \eqref{eq:condJac}, $b_\eps$ satisfies the weaker condition
			\begin{equation}\label{weak_Jac}
				\begin{split}
					\mathrm{Jac}\,b_\eps(x)\ge -Cd(x)^{-2}\kappa(x)I_{n\times n}\,,\qquad \forall\,x\in\Gamma_{\delta_0}\,,\forall\,\eps>0\,,
				\end{split}
			\end{equation}
			where $\kappa(x)\to0$ when $x\to\partial\Omega$. Then, if the stronger invariance condition \eqref{inv_stronger} is satisfied for some $\beta_0>1$, we have for every $a>0$:
			\begin{equation}\label{holdest}
				\begin{split}
					\norm{d^a(\cdot)\nabla\phi_\eps(\cdot)}_{L^\infty(\Omega_\eps)}\le K_a\norm{g_\eps}_{W^{1,\infty}(\Omega_\eps)}\,,
				\end{split}
			\end{equation}
			for some $K$ depending on $a,\Omega$ and $b$, but independent on $\eps$.
		\end{prop}
		
		\begin{rem} We observe that condition \eqref{eq:condJac} is satisfied at least in our framework, for $p\neq2$. Actually, since $b(x)=p|\nabla u|^{p-2}\nabla u=-\alpha(x)$, we have
			$$
			\mathrm{Jac}\,b(x)=p|\nabla u|^{p-2}D^2u+p(p-2)|\nabla u|^{p-4}D^2u Du\otimes Du\,.
			$$
			Using \eqref{eq:stimeDu} and \eqref{eq:taylorD2u} for $p\neq 2$, one has
			\begin{equation*}
				\begin{split}
					\mathrm{Jac}\,b(x)=\frac{p'}{d(x)^2}\big[\nu(x)\otimes\nu(x)+O(\omega_d)\big]\,.
				\end{split}
			\end{equation*}
			Now, for each $\xi\in\R^n$ we have
			$$
			\langle \nu(x)\otimes\nu(x)\xi,\xi\rangle=|\langle\xi,\nu\rangle|^2\ge0\,.
			$$
			This implies, for $p\neq 2$,
			\begin{equation}\label{eq:Jacc}
				\begin{split}
					\langle(\mathrm{Jac}\,b)\,\xi,\xi\rangle=\frac{p'}{d(x)^2}\big[|\langle\xi,\nu\rangle|^2+O(\omega_d)\big]\ge -C_pd(x)^{\sigma_p-2}\,,
				\end{split}
			\end{equation}
			where $\sigma_p>0$, as can be seen from the definition of $\omega_d$ in \eqref{def:omega}.
			
			For $p=2$, the weaker condition \eqref{weak_Jac} can be obtained, using \eqref{eq:stimeD2u} instead of \eqref{eq:taylorD2u}:
			\begin{equation}\label{eq:Jacc2}
				\begin{split}
					\langle(\mathrm{Jac}\,b)\,\xi,\xi\rangle=\frac{2}{d(x)^2}\big[|\langle\xi,\nu\rangle|^2+o(1)\big]\ge-d(x)^2\kappa(x)\,,
				\end{split}
			\end{equation}
			where $\kappa(x)\to0$ when $x\to\partial\Omega$.
		\end{rem}
		
		\vskip1em
		The proof of Proposition \ref{prop:Lip} will rely on the following strong maximum principle.
		
		\begin{lem}\label{smp} Let $\Omega$ be  a connected domain and assume that $b\in W^{1,\infty}_{loc}(\Omega)$ satisfies \eqref{eq:inv}. Then any function $v$ which is a weak solution of
			$$
			\begin{cases}
				-\Delta  v+b (x)\cdot \nabla  v  =0  &   \hbox{in $\Omega$}\,,\\
				v\in L^\infty(\Omega) & 
			\end{cases}
			$$
			is constant in $\Omega$.
		\end{lem}
		
		\begin{proof} Up to  adding a constant to $v$, we can assume that $v\geq 0$.  Let us take $\vep>0$, sufficiently small; the function $v$ is clearly a subsolution of the problem
			$$
			-\Delta  v+b (x)\cdot \nabla  v + \vep v  = \vep \|v\|_\infty   \qquad  \hbox{in $\Omega\setminus \Omega_{\de_0}$.}
			$$
			Now we  consider $V:= M_0-  \vep \log(d(x))$, where $M_0:= \sup\limits_{\{d(x)=\de_0\} } v$. Using \rife{eq:inv} we have
			\begin{align*}
				-\Delta  V & +b (x)\cdot \nabla  V + \vep V  - \vep \|v\|_\infty  \\ 
				& =  \frac\vep{d(x) } \left(-\frac1{d(x) }- b(x) \cdot \nabla d(x) +\Delta d(x) - d(x) \log (d(x)) + d(x)\Big[ M_0  - \|v\|_\infty\Big]\right) \\
				& \geq \frac\vep{d(x) } \big( - d(x) \log (d(x)) - C d(x)\big)
			\end{align*}
			hence we get  
			$$
			-\Delta  V+b (x)\cdot \nabla  V + \vep V  \geq  \vep \|v\|_\infty   \qquad  \hbox{in $\Omega\setminus \Omega_{\de_0}$,}
			$$
			provided $\de_0$ is sufficiently small. Since $V\to +\infty$ as $x\to \partial \Omega$, and $V\geq v$ on $\{d(x)=\de_0\}$, by comparison principle we deduce that $v\leq V$ in $\Omega\setminus \Omega_{\de_0}$. This means $v(x) \leq  \sup\limits_{\{d(x)=\de_0\} } v-  \vep \log(d(x))$; letting $\vep \to 0$ yields
			$$
			\sup\limits_{\{d(x)<\de_0\} } v \leq  \sup\limits_{\{d(x)=\de_0\} } v\,.
			$$
			Since $b$ is locally bounded, we easily get $v(x) \leq \sup\limits_{\{d(x)=\de_0\} } v$ for every $x\in \Omega_{\de_0}$, by usual maximum  principle. Hence, we obtain that
			$$
			\sup\limits_{\Omega } v = \max\limits_{\{d(x)=\de_0\} } v\,.
			$$
			This means that $v$ attains a global  maximum at some point $x_0$ inside $\Omega$. Being $b$ smooth in $\Omega$, the strong maximum principle applies and we deduce that $v$ is constant in $\Omega$ (which is assumed to be connected).
		\end{proof}
		
		Now we are ready to give the proof of Proposition \ref{prop:Lip}.
		\vskip2em
		
		\begin{proof}[Proof of Proposition \ref{prop:Lip}]
			
			\emph{Step 1.} Let $\phi_\eps$ be a solution of \eqref{eq:phieps}, where we assume that $\de_\eps\geq 0$ and $g_\eps$ satisfies $\|  g_\eps\|_{W^{1,\infty}(\Omega_\eps)}\leq 1$.
			% where we define $\nabla\phi_\eps(x)=0$ if $x\in\Gamma_\eps$. 
			We begin by proving that, if \eqref{eq:inv} and \eqref{eq:condJac} hold, then there exist $\delta_0, C_0>0$, independent on $\eps$, such that
			\begin{equation}\label{eq:stimacorona}
				\sup\limits_{x\in\Gamma_{\delta_0}}|\nabla\phi_\eps(x)|\le  C_0+ \sup\limits_{x\in\overline \Omega_{\delta_0}}|\nabla\phi_\eps(x)|\,,
			\end{equation}
			whereas, if \eqref{inv_stronger} and \eqref{weak_Jac} hold, then $\forall\beta>0$ there exist $\delta_0=\delta_0(\beta)$, $C_0=C_0(\beta)$, both independent on $\eps$, such that
			\begin{equation}\label{weak_stimacorona}
				\sup\limits_{x\in\Gamma_{\delta_0}}\big[d^\beta(x)|\nabla\phi_\eps(x)|\big]\le  C_0\left(1+ \sup\limits_{x\in\overline \Omega_{\delta_0}}|\nabla\phi_\eps(x)|\right)\,.
			\end{equation}
			To this purpose, we argue similarly as in \cite{Por_Leon}; we consider the function $w_\eps=|\nabla\phi_\eps|^2e^{\theta(d)}$, where $\theta(\cdot):(0,+\infty)\to\R$ is a smooth function bounded from above, which will be defined later. The gradient and the Laplacian of $w_\eps$ are computed as follows:
			\begin{equation*}
				\begin{split}
					\nabla w_\eps &= 2e^{\theta(d)}\nabla\phi_\eps D^2\phi_\eps+w_\eps\theta'(d)\nabla d\,,\\
					\Delta w_\eps &= 2e^{\theta(d)}\left[\mathrm{tr}(D^2\phi_\eps D^2\phi_\eps)+\nabla\phi_\eps\cdot\nabla(\Delta\phi_\eps)+\theta'(d)(D^2\phi_\eps\nabla\phi_\eps)\cdot\nabla d\right]\\
					&+\theta'(d)\nabla w_\eps\cdot\nabla d+w_\eps\theta''(d)|\nabla d|^2+w_\eps\theta'(d)\Delta d\,.
				\end{split}
			\end{equation*}
			For $\delta_0$ sufficiently small we have $|\nabla d|=1$. Hence, using \eqref{eq:phieps}, we easily obtain the equation satisfied by $w_\eps$ in $\Gamma_{\delta_0}$:
			\begin{equation*}
				\begin{split}
					-\Delta w_\eps(x) +B_\eps(x)\cdot\nabla w_\eps(x)+ C_\eps(x)w_\eps(x)=G_\eps(x)\,,
				\end{split}
			\end{equation*}
			where
			\begin{equation*}
				\begin{split}
					B_\eps&=b_\eps+2\theta'(d)\nabla d\,,\\
					C_\eps&=\delta_\eps+(\Delta d-b_\eps\nabla d)\theta'(d)+(\theta''(d)-\theta'(d)^2)\,,\\
					G_\eps&=2e^{\theta(d)}\left[-\mathrm{tr}(D^2\phi_\eps D^2\phi_\eps)-(\mathrm{Jac}\,b_\eps)\nabla\phi_\eps\cdot\nabla\phi_\eps+\nabla g_\eps\cdot\nabla\phi_\eps\right]\,.
				\end{split}
			\end{equation*}
			In the first case, using the bound on $\nabla g_\eps$ and assumption \eqref{eq:condJac}, we easily get $G_\eps\le C(1+d^{\gamma_0-2}w_\eps)$, for a certain $C>0$. Hence, the equation on $w_\eps$ becomes
			\begin{equation*}
				\begin{split}
					-\Delta w_\eps(x)+B_\eps(x)\cdot\nabla w_\eps(x)+(C_\eps(x)-Cd^{\gamma_0-2})w_\eps(x)\le C\,.
				\end{split}
			\end{equation*}
			We choose $\theta(d)=d^\gamma$, with $0<\gamma<\gamma_0$. Using \eqref{eq:inv}, we can bound from below the coefficient of $w_\eps$:
			\begin{equation*}
				\begin{split}
					C_\eps(x)-Cd^{\gamma_0-2}&=\delta_\eps-Cd^{\gamma_0-2}+\gamma d^{\gamma-1}(\Delta d-b_\eps\cdot\nabla d)+\gamma(\gamma-1)d^{\gamma-2}-\gamma^2 d^{2\gamma-2}\\
					&\ge\gamma^2 d^{\gamma-2}(1+o(1))>0\,,
				\end{split}
			\end{equation*}
			provided $\delta_0$ is small enough. Hence, the maximum principle implies
			$$
			\sup\limits_{x\in\Gamma_{\delta_0}}w_\eps\le C+ \max\limits_{x\in\partial\left(\Gamma_{\delta_0}\cap\Omega_\eps\right)}w_\eps\,.
			$$
			The maximum on the right-hand side cannot be attained at some point $x_0\in\{d(x)=\eps\}$, since  we have $\partial_\nu w_\eps < 0$ on $\{d(x)=\eps\}$ thanks to the Neumann condition on $\phi_\eps$ (see e.g. \emph{Lemma 4} of \cite{Por_Leon}). Hence, the maximum is attained in $\{d(x)=\delta_0\}\subset \overline \Omega_{\delta_0}$, which proves \eqref{eq:stimacorona}.
			
			Now, assume that \eqref{inv_stronger} and \eqref{weak_Jac} hold. In that case, the bound on $G_\eps$ becomes $G_\eps\le C(1+d^{-2}\kappa w_\eps)$. Hence, the equation on $w_\eps$ becomes
			$$
			-\Delta w_\eps(x)+B_\eps(x)\cdot\nabla w_\eps(x)+(C_\eps(x)-C\kappa d^{-2})w_\eps(x)\le C\,.
			$$
			For $0<\gamma<\beta_0-1$, we choose $\theta(d)=\gamma\log d$, which means $e^{\theta(d)}=d^\gamma$. Using \eqref{inv_stronger}, we have the following bound for the coefficient of $w_\eps$:
			\begin{equation*}
				\begin{split}
					C_\eps(x)- C\kappa d^{-2}&=\delta_\eps-C-C\kappa d^{-2}+\gamma d^{-1}(\Delta d-b_\eps\cdot\nabla d)-\gamma d^{-2}-\gamma^2 d^{-2}\\
					&\ge d^{-2}\big(\gamma(\beta_0-1-\gamma)+o(1)\big)>0\,,
				\end{split}
			\end{equation*}
			provided $\delta_0$ (depending on $\gamma$) is small enough. We can apply the maximum principle, and as before we get
			$$
			\sup\limits_{x\in\Gamma_{\delta_0}}d^\gamma|\nabla\phi_\eps|^2=\sup\limits_{x\in\Gamma_{\delta_0}}w_\eps\le C+\sup\limits_{x\in\partial{\Omega}_{\delta_0}}w_\eps\le C+\delta_0^\gamma\sup\limits_{x\in\overline{\Omega}_{\delta_0}}|\nabla\phi_\eps|^2\,,
			$$
			which implies \eqref{weak_stimacorona}.\\
			
			\emph{Step 2.} To conclude, we claim that, calling $K_0$ the compact set $\overline \Omega_{\de_0}$, we have 
			\begin{equation}\label{eq:grad_int}
				\max\limits_{x\in K_0}|\nabla\phi_\eps(x)|\le \tilde C_0\,,
			\end{equation}
			for all $\eps>0$ and for a certain constant $\tilde C_0>0$.
			
			To show  \rife{eq:grad_int}, we argue by contradiction, and we suppose that, for some sequence (not relabeled), we have
			$$
			\|\nabla\phi_\eps\|_{L^\infty(K_0)} \to \infty\,.
			$$
			Let us suppose for now that $\de_\vep\to 0$, which means that we are considering  the ergodic problem. For a fixed $x_0\in\Omega$, we define 
			$$
			\hat \phi_\eps:=
			% \begin{cases}
				\frac{\phi_\eps(x)-\phi_\eps(x_0)}{\|\nabla\phi_\eps\|_{L^\infty(K_0)} } \,.
				%& \hbox{if $\de_\vep\to 0$ (ergodic problem)}\\
				%\frac{\phi_\eps(x)}{\|\nabla\phi_\eps\|_{L^\infty(K_0)} }& \hbox{if $\de_\vep\to \de>0$}\,.
				%\end{cases}
				$$
				Rescaling \rife{eq:phieps}, $\hat \phi_\eps$ satisfies
				\begin{equation}\label{cappuc}
					\begin{cases}
						-\Delta\hat \phi_\eps(x)+b_\eps(x)\cdot  \nabla\hat \phi_\eps(x)+\delta_\eps\hat \phi_\eps(x)=\frac{g_\eps(x)}{\|\nabla\phi_\eps\|_{L^\infty(K_0)}} - \frac{\delta_\eps \phi_\eps(x_0)}{\|\nabla\phi_\eps\|_{L^\infty(K_0)}} \,,\quad x\in \Omega_\eps\,,\\
						\nabla\hat \phi_\eps(x)\cdot\nu_\eps(x)_{|\partial\Omega_\eps}=0\,.
					\end{cases}
				\end{equation}
				Notice that, by maximum principle, we have $\|\delta_\eps \phi_\eps\|_\infty\leq \|g_\vep \|_\infty \leq C$; hence, the right-hand side in the previous equation vanishes as $\eps\to 0$ (because $\|\nabla\phi_\eps\|_{L^\infty(K_0)}\to \infty$).
				
				By Step 1, we have that either \eqref{eq:stimacorona} or \eqref{weak_stimacorona} holds. In the first case, we have 
				$$
				\sup\limits_{x\in\Gamma_{\delta_0}\cap \Omega_\vep} | \nabla \hat \phi_\eps|   
				%C_0 + \sup\limits_{x\in\partial K_0} | \nabla \hat \phi_\eps| 
				\leq C_0+1
				$$
				and since $\| \nabla \hat \phi_\eps\|_{L^\infty(K_0)}=  1$ we deduce that $\nabla \hat \phi_\eps$ is uniformly bounded in $\Omega_\eps$. Since $\hat \phi_\eps(x_0)=0$, we conclude  that $\hat \phi_\eps$ is equibounded and uniformly Lipschitz continuous in $\Omega_\eps$.
				
				If \eqref{weak_stimacorona} holds, let $a\in(0,1)$. For $x\in\Gamma_{\delta_0}$ we have, for a $C$ depending on $\beta$,
				$$
				|\nabla\phi_\eps(x)|\le C d^{-\beta}(x)\left(1+\norm{\nabla\phi_\eps}_{L^\infty(K_0)}\right)\implies|\nabla\hat\phi_\eps(x)|\le Cd^{-\beta}(x)\,.
				$$
				Since $\| \nabla \hat \phi_\eps\|_{L^\infty(K_0)}=1$, we deduce that $\nabla\hat\phi_\eps\in L^r(\Omega_\eps)$ uniformly in $\eps$ and for $r<\frac1\beta$. Moreover, since $d^{-\beta}$ is integrable for small $\beta$, and $\hat\phi_\eps(x_0)=0$, we have that $\hat\phi_\eps$ is equibounded, and so $\hat\phi_\eps\in W^{1,r}(\Omega_\eps)$ uniformly in $\eps$. Choosing any $\beta<n^{-1}(1-a)$, this implies that $\hat\phi_\eps$ is uniformly $a$-H\"older continuous in $\Omega_\eps$.
				
				In both cases, we can apply Ascoli-Arzel\'a's theorem, which ensures the locally uniform convergence (up to subsequences) of $\hat \phi_\eps$ in $\Omega$ to some local Lipschitz function $\hat \phi$; in addition, by ellipticity, the convergence holds (at least) in $C^1$ for all compact subsets  of $\Omega$. Therefore, once we pass to the limit, $\hat \phi$ satisfies
				$$
				\begin{cases}
					-\Delta\hat \phi +b (x)\cdot \nabla\hat \phi=0\,,\\
					\hat \phi\in L^\infty(\Omega)\cap C^1 (\Omega)\,.
				\end{cases}
				$$
				In particular, $\hat\phi$ is a weak solution of the above equation, in the sense of Definition \ref{defhjb2}. 
				Applying now Lemma \ref{smp}, we deduce that $\hat \phi$ is constant. But since   $\| \nabla \hat \phi_\eps\|_{L^\infty(K_0)} =1$ and $\hat \phi_\eps\to \hat \phi$ in $C^1(K_0)$, we also have $\| \nabla \hat \phi \|_{L^\infty(K_0)} =1$. This is a contradiction. We conclude that \rife{eq:grad_int} holds true; and together with \rife{eq:stimacorona} and \rife{weak_stimacorona}, this concludes the proof. In the case that $\de_\eps$ converges (up to subsequences) towards some $\de>0$, we argue in the same way but we simply define $\hat \phi_\eps= \frac{\phi_\eps(x)}{\|\nabla\phi_\eps\|_{L^\infty(K_0)} }$. The conclusion follows as before.
				Finally, once the result is proved assuming $\| g_\eps\|_{W^{1,\infty}(\Omega_\eps)}\leq 1$, by linearity we get the estimates \rife{gradest} and \rife{holdest}.
			\end{proof}

		}
		
		As an immediate corollary, we have the following regularity estimate for the problem \eqref{hjb2} with $\delta>0$.
		\begin{cor}
			Let $\delta>0$, $g\in W^{1,\infty}(\Omega)$, $b\in W^{1,\infty}_{loc}(\Omega)$, satisfying \eqref{eq:inv} and \eqref{eq:condJac}. Then the solution $\phi$ of \eqref{hjb2} belongs to $W^{1,\infty}(\Omega)$.
			
			If the coefficients satisfy \eqref{inv_stronger} and \eqref{weak_Jac}, then the solution $\phi$ belongs to $C^a(\Omega)$ for every $0<a<1$.
			\begin{proof}
				We already know that  $\phi\in L^\infty$. For the bound of the gradient, it suffices to note that $\phi$ is the limit of $\phi_\eps$ which solves \eqref{eq:lella}. Since $g\in W^{1,\infty}(\Omega)$, we can apply to \eqref{eq:lella} the conclusion of Proposition \ref{prop:Lip}. Then $\phi_\eps$ is uniformly bounded in $W^{1,\infty}$ if \eqref{eq:inv} and \eqref{eq:condJac} holds, and in $C^a$ for every $a\in(0,1)$ if \eqref{inv_stronger} and \eqref{weak_Jac} hold. Passing to the limit for $\eps\to 0$ the same estimates hold for $\phi$.
			\end{proof}
		\end{cor}
		Now we focus on the case $\delta=0$. The main difference here is that the problem \eqref{eq:lella} is not defined for every $g$, and we have to make use of Fredholm's alternative, introducing the ergodic constant. 
		%Even the uniqueness fails to be true.
		
		\begin{prop}\label{prop:delta0}
			Let $g\in W^{1,\infty}(\Omega)$, $b\in W^{1,\infty}_{loc}(\Omega)$ such that \eqref{eq:inv} and \eqref{eq:condJac} (resp. \eqref{inv_stronger} and \eqref{weak_Jac}) are satisfied.  Then there exists a unique
			$\lambda_g\in\R$ such that the problem
			\begin{equation}\label{eq:hjb0}
				-\Delta\phi(x)+b(x)\cdot\nabla\phi(x)=g(x)+\lambda_g
			\end{equation}
			admits a weak solution $\phi\in L^\infty(\Omega)$. Moreover, $\phi$ is unique up to an additive constant, and belongs to  $W^{1,\infty}(\Omega)$ (resp. to $C^a(\Omega)$ for all $0<a<1$ if \eqref{inv_stronger} and \eqref{weak_Jac} hold). 
		\end{prop}
		When there is no possibility of mistake, we will write $\lambda$ instead of $\lambda_g$.
		\begin{proof}
			We consider the problem \eqref{eq:lella} with $\delta=0$. It is a well-known fact (see for example \cite{Krylov}, \cite{Droniou} and \cite{GT}) that the adjoint problem
			\begin{equation*}
				\begin{cases}
					-\Delta\rho_\eps(x)-\mathrm{div}\big(\rho_\eps(x) b(x)\big)=0\,,\\
					\big[\nabla\rho_\eps(x)+\rho_\eps(x) b(x)\big]\cdot\nu(x)_{|\partial\Omega_\eps}=0
				\end{cases}
			\end{equation*}
			admits a solution $\rho_\eps\ge0$, unique up to a multiplicative constant. Hence, Fredholm's alternative tells us that there exists a unique $\lambda_\eps$ and a function $\phi_\eps$, unique up to an additive constant, which solves
			\begin{equation*}
				\begin{cases}
					-\Delta\phi_\eps(x)+b(x)\cdot\nabla\phi_\eps(x)=g(x)+\lambda_\eps\,,\quad x\in \Omega_\eps\\
					\nabla\phi_\eps(x)\cdot\nu(x)_{|\partial\Omega_\eps}=0\,.
				\end{cases}
			\end{equation*}
			We fix $x_0\in\Omega$ and we choose $\phi_\eps$ and $\rho_\eps$ with the additional conditions
			$$
			\phi_\eps(x_0)=0\,,\qquad \into \rho_\eps(x)\,dx=1\,.
			$$
			Multiplying the equation of $\rho_\eps$ by $\phi_\eps$ and integrating by parts, we get
			$$
			\into\rho_\eps(x)(g(x)+\lambda_\eps)\,dx=0\,,
			$$
			which immediately implies
			\begin{equation}\label{eq:boundlambda}
				\begin{split}
					|\lambda_\eps|=\left|\into\rho_\eps(x)g(x)\,dx\right|\le\norminf{g}\,.
				\end{split}
			\end{equation}
			Hence, $\{\lambda_\eps\}_\eps$ is uniformly bounded, and, up to a not-relabeled subsequence, there exists $\lambda$ such that $\lambda_\eps\to\lambda$, and $|\lambda|\le\norminf{g}$.
			
			Now we apply Proposition \ref{prop:Lip} with $b_\eps=b$, $\delta_\eps=0$, $g_\eps(x)=g(x)+\lambda_\eps$. This ensures the $L^\infty$ (resp. $L^r$ for all $r$) bound of $\nabla\phi_\eps$, which, together with the condition $\phi_\eps(x_0)=0$, implies the uniform $W^{1,\infty}$ (resp. $W^{1,r}$) bound of $\phi_\eps$. Hence there exists $\phi\in W^{1,\infty}(\Omega)$ (resp. $ W^{1,r}$) such that, up to a not-relabeled subsequence,
			$$
			\phi_\eps\to\phi\quad\mbox{in }C(\Omega)\,,\qquad\nabla\phi_\eps\weak\nabla\phi\quad\mbox{in }L^r(\Omega)\quad\forall\,r\ge1\,.
			$$
			Passing to the limit in the weak formulation of $\phi_\eps$, we obtain that $\phi$ is a weak solution of \eqref{eq:hjb0}, in the sense of Definition \ref{defhjb2}.
			
			Now we prove the uniqueness of $\lambda $, arguing as in Lemma \ref{smp}.  Suppose   that $(\lambda_1,\phi_1)$ and $\lambda_2, \phi_2$ are two different solutions with, say, $\lambda_1>\lambda_2$. Then, for a  sufficiently small $\vep, \de_0>0$, we can see that  the function $\Phi:= \phi_1- \vep \log (d(x) ) +    \sup\limits_{\{d(x)=\de_0\} } (\phi_2-\phi_1)$ satisfies
			$$
			-\Delta\Phi(x)+b(x)\cdot\nabla\Phi (x)> g(x)+\lambda_2 \quad \hbox{for $x\in \Omega\setminus \Omega_{\de_0}$.}
			$$
			Hence by the maximum principle we deduce that  $\phi_2 \leq \Phi$ in $\Omega\setminus \Omega_{\de_0}$, and letting $\vep \to 0$ yields 
			$$
			\sup\limits_{\{d(x)<\de_0\} } (\phi_2-\phi_1) = \sup\limits_{\{d(x)=\de_0\} } (\phi_2-\phi_1)
			\,.
			$$
			Inside $\Omega_{\de_0}$, using $\lambda_1>\lambda_2$ we also have $\sup\limits_{\Omega_{\de_0} } (\phi_2-\phi_1)= \sup\limits_{\{d(x)=\de_0\} } (\phi_2-\phi_1)$. This allows us to conclude that $ (\phi_2-\phi_1)$ has a global maximum attained on $\{d(x)=\de_0\}$, hence it is a local maximum  inside the domain $\Omega$.  But this is impossible by maximum principle, since $\lambda_1 >\lambda_2$ implies $-\Delta(\phi_2-\phi_1)+b(x)\cdot\nabla(\phi_2-\phi_1) <0$. The contradiction shows that $\lambda_1= \lambda_2$. Finally, suppose that $\phi_1,\phi_2$ are two different weak solutions corresponding to $\lambda$; applying Lemma \ref{smp} we deduce that $\phi_1-\phi_2$ is constant in $\Omega$.

		\end{proof}
		\section{The Fokker-Planck equation}
		
		Now we focus on the study of the following Fokker-Planck equation:
		\begin{equation}\label{eq:fp}
			-\Delta m -\mathrm{div}(m\,b(x))=0\,,\qquad x\in\Omega\,.
		\end{equation}
		
		where $b:\Omega\to\R$ is a locally bounded vector field.
		
		As in the previous section,  \emph{no boundary condition is prescribed} for the following equation. In place of that, we require the invariance condition \eqref{eq:inv}.
		
		\begin{comment}
			\begin{equation}\label{eq:condfpweak}
				\lim\limits_{x\to\partial\Omega}b(x)\cdot\nabla d(x)=-\infty\,.
			\end{equation}
			
			We note that this condition is satisfied in our framework. Actually, in the Mean Field Games system we have $b(x)=p|\nabla u|^{p-2}\nabla u=-\alpha(x)$ and so, by \eqref{eq:stimeDu},
			$$
			\lim\limits_{x\to\partial\Omega}b(x)\cdot \nabla d(x)=\lim\limits_{x\to\partial\Omega}\frac{p'}{d(x)}\,\nu(x)\cdot \nabla d(x)+\psi(x)\nabla d(x)=-\infty\,,
			$$
			since, for the property of the distance function (see \cite{DistFunct}), $\nabla d(x)=-\nu(x)$ for $x\in\partial\Omega\,.$
		\end{comment}
		
		We start by giving a suitable definition  of solution for the problem \eqref{eq:fp}. We denote by $C_b(\Omega)$ the space of continuous bounded functions in $\Omega$.
		
		\begin{defn}\label{def:fp}
			We say that $m\in\mathcal{P}(\Omega)$ is a weak solution of  the equation \eqref{eq:fp} if, for all $\phi\in L^\infty(\Omega)$ satisfying
			\begin{equation}\label{eq:testfp}
				-\Delta\phi+b\cdot\nabla\phi\in C_b(\Omega)
			\end{equation}
			in the sense of Definition \ref{defhjb2}, the following equality holds:
			\begin{equation}\label{eq:fpformula}
				\into \Big(-\Delta\phi(x)+b(x)\cdot\nabla\phi(x)\Big)\,dm(x)=0\,.
			\end{equation}
		\end{defn}
		
		Now we provide with   existence and uniqueness of solutions in the above sense.
		\begin{thm}\label{thm:exm}
			Suppose that $b\in W^{1,\infty}_{loc}(\Omega;\R^n)$ and satisfies \eqref{eq:inv} and \eqref{eq:condJac}, or \eqref{inv_stronger} and \eqref{weak_Jac}. Then there exists a unique probability measure $m\in\mathcal{P}(\Omega)$, absolutely continuous with $L^1$ density, which satisfies the equation \eqref{eq:fp}.
		\end{thm}
		
		With an abuse of notation, we will call $m$ both the probability measure and its $L^1$ density.
		
		\begin{proof}
			Let $\eps>0$. Given, as before,  $\Omega_{\eps}:=\{d(x)>\eps\}\subset\Omega$, we define $m_\eps$ as the unique solution of the following elliptic PDE:
			\begin{equation}\label{eq:m_epsapprox}
				\begin{cases}
					-\Delta m_\eps-\mathrm{div}(m_\eps\,b(x))=0\,,\hspace{0.2cm}\qquad x\in\Omega_\eps\,,\\
					\left[\nabla m_\eps+m_\eps\,b(x)\right]\cdot\nu(x)=0\,,\qquad x\in\partial\Omega_\eps\,,\\
					m_\eps\ge0\,,\qquad \mathlarger{\into} m_\eps(x)=1\,.
				\end{cases}
			\end{equation}
			We can eventually extend  $m_\eps$ to the whole $\Omega$ by putting $m_\eps(x)=0$ for $x\in\Omega\setminus\Omega_\eps\,.$
			Since $b$ is locally bounded, according to standard estimates (e.g. see \emph{Corollary 1.6.4} in  \cite{Krylov}),   for all compact subsets $K\subset\subset\Omega$ we have
			$$
			\int_{\Omega_\eps} \big(m_\eps^\alpha+|\nabla m_\eps|^\alpha\big)\,dx\le C_K\,,
			$$
			for some $\alpha>1$ and for $C_K$ not depending on $\eps$. This implies that $\exists\,m\ge0$, $m\in W^{1,\alpha}_{loc}(\Omega)$, such that, up to subsequences,
			\begin{equation*}
				m_\eps\to m\qquad \hbox{$a.e.$ and in }L^\alpha_{loc}(\Omega)\,,\qquad\qquad m_\eps\weak m\qquad\mbox{in }W^{1,\alpha}_{loc}(\Omega)\,.
			\end{equation*}
			It is straightforward that $m\ge0\,\, a.e.$.
			We now want to prove that there is no dissipation of mass, and actually $m_\eps\to m$ strongly in $L^1$. Indeed, since we already know the a.e. convergence, we just have to prove that
			\begin{equation*}
				\into m_\eps(x)\, dx\to \into m(x)\, dx\,.
				%\implies \into m(x)\, dx=1\implies \into m(x)\, dx\ge 1\,,
			\end{equation*}
			Being $m_\eps\in \cP(\Omega)$, this reduces to show that $\into m(x)\, dx=1$, and we only need to prove that $\into m(x)\, dx\ge 1$, 
			since, by Fatou's lemma, we already have $\into m(x)\,dx\le 1\,$.\\
			
			In order to do that, we use as test function in \eqref{eq:m_epsapprox} a $\mathcal{C}^2$ approximation of the oriented distance $d(\cdot)$, coinciding with $d(\cdot)$ in a neighborhood of the boundary. To avoid too heavy notation, we call this approximation $d(\cdot)$ as well. We obtain, for $\eps$ sufficiently small,
			\begin{equation}\label{eq:partintm}
				\int_{\Omega_\eps} m_\eps(-\Delta d(x)+b(x)\cdot\nabla d(x))\,dx=-\int_{\partial\Omega_\eps}m_\eps\nabla d(x)\cdot\nu_\eps(x)\,dx\ge0\,,
			\end{equation}
			where $\nu_\eps(x)$ is the outward normal at $\partial\Omega_\eps$ and the inequality in the right-hand side comes from the fact that $\nu_\eps(x)=\nu(x) =-\nabla d(x)$  for $\eps$ sufficiently small.  
			%$d(x)=d_{\Omega_\eps}(x)+\eps$ near $\partial\Omega$, which implies $\nabla d\cdot\nu_\eps=\nabla d_{\Omega_\eps}\cdot\nu_{\eps}=-|\nu_\eps|^2\le 0\,.$
			
			Let $\delta>0\,$, which will be chosen later. From \eqref{eq:partintm} we get, splitting the integral in the left-hand side in $\Omega_\delta$ and $\Omega_{\eps}\setminus\Omega_\delta$, and recalling that $m_\eps=0$ in $\Omega\setminus\Omega_\eps$,
			\begin{equation}\label{eq:intme}
				\int_{\Omega\setminus\Omega_\delta} m_\eps \big(\Delta d(x)-b(x)\cdot\nabla d(x)\big)\,dx\le \int_{\Omega_\delta} m_\eps \big(-\Delta d(x)+b(x)\cdot\nabla d(x)\big)\le C_\delta\,,
			\end{equation}
			where the last bound is because $b\in L^\infty_{loc}(\Omega;\R^n)$ and $m_\eps$ has unit mass.  
			
			Thanks to \rife{eq:inv}, $\lim\limits_{x\to\partial\Omega}\big(\Delta d(x)-b(x)\cdot\nabla d(x)\big)=+\infty\,$. Hence, we can choose $\delta$ such that  $\Delta d(x)-b(x)\cdot\nabla d(x)\ge0$ in $\Omega\setminus\Omega_\delta$.
			
			Now, let $0<\eta<\delta\,.$ Since $\Omega\setminus\Omega_\eta\subseteq\Omega\setminus\Omega_\delta$, and the quantity in the integral on the left-hand side of \eqref{eq:intme} is non-negative, we get
			\begin{equation*}
				\int_{\Omega\setminus\Omega_\eta} m_\eps \big(\Delta d(x)-b(x)\cdot\nabla d(x)\big)\,dx\le C_\delta\,.
			\end{equation*}
			For each $N>0$, we choose $\eta=\eta_N$ such that
			$$
			\Delta d(x)-b(x)\cdot\nabla d(x)\ge N\qquad\quad\forall\, x\in\Omega\setminus\Omega_\eta\,.
			$$
			Hence we obtain
			\begin{equation}\label{mass-eta}
				N \int_{\Omega\setminus\Omega_\eta} m_\eps\,dx\le C_\delta
				%\implies\int_{\Omega\setminus\Omega\eta} m_\eps\,dx\le \frac{C_\delta}N\implies\int_{\Omega\eta}m_\eps\,dx\ge1-\frac{C_{\delta}}{N}\,.
			\end{equation}
			which yields
			\begin{equation}\label{mass-eta}
				\int_{\Omega_\eta}m_\eps\,dx\ge1-\frac{C_{\delta}}{N}\,.
			\end{equation}
			Since the last estimate, for fixed $N$, is uniform in $\eps$ and is referred to the compact set $\Omega_\eta$, we can use the local convergence of $m_\eps\to m$ and obtain
			\begin{align*}
				\into m(x)\,dx\ge\int_{\Omega_\eta}m(x)\,dx=\lim\limits_{\eps\to0}\int_{\Omega_\eta}m_\eps\,dx\ge1-\frac{C_{\delta}}{N}\,.
			\end{align*}
			Eventually, passing to the limit for $N\to+\infty$, we obtain
			$$
			\into m(x)\,dx\ge 1\,.
			$$
			Hence we conclude that 
			$$
			\into m(x)\,dx=1\,,
			$$
			as well as  the strong convergence  $m_\eps\to m$ in $L^1$.
			
			To prove equation \eqref{eq:fpformula}, we take $\phi\in L^{\infty}(\Omega)$ solving \eqref{eq:testfp} and we consider the sequence $\{\phi_\eps\}\subseteq L^\infty(\Omega)$ such that $\phi_\eps$ solves in $\Omega_\eps$
			$$
			\begin{cases}
				-\Delta\phi_\eps+b\cdot\nabla\phi_\eps+\phi_\eps=-\Delta\phi+b\cdot\nabla\phi+\phi\,,\\
				\nabla\phi_\eps\cdot\nu(x)_{|\partial\Omega_\eps}=0\,.
			\end{cases}
			$$
			Thanks to Proposition \ref{prop:hjdelta} we know that $\phi_\eps\weak^*\phi$, with a $L^\infty$-weak$^*$ convergence. 
			
			We use $\phi_\eps$ as test function for $m_\eps$, obtaining
			$$
			\int_{\Omega_\eps} m_\eps(x)\big(-\Delta \phi(x)+b(x)\cdot\nabla \phi(x)\big)\,dx=\int_{\Omega_\eps}m_\eps(x)\big(\phi_\eps(x)-\phi(x)\big)\,dx\,.
			$$
			Since $-\Delta\phi+b\cdot\nabla\phi\in L^\infty$, while $m_\eps\to m$ in $L^1$, we can pass to the limit on the left-hand side. Similarly we have for the right-hand side, using that $\phi_\eps\weak^*\phi$ in $L^\infty$. This allows us  to pass to the limit and conclude the existence part.\\
			
			As far as uniqueness is concerned, let us consider $m_1$ and $m_2$ two solutions of \eqref{eq:fp}, and take $f\in W^{1,\infty}(\Omega)$. According to Proposition \ref{prop:delta0}, there exist $\lambda_f\in\R$, $\phi\in W^{1,r}(\Omega)$, $1\le r<+\infty$, such that
			$$
			-\Delta\phi+b\cdot\nabla\phi=f+\lambda_f\,,
			$$
			in the sense of Definition \ref{defhjb2}. Using $\phi$ as test function in \eqref{eq:fpformula}, we get
			$$
			\into (f(x)+\lambda_f)dm_1(x)=\into (f(x)+\lambda_f)dm_2(x)=0
			$$
			which implies
			$$
			\into f(x)dm_1(x)=\into f(x)dm_2(x)\,.
			$$
			The arbitrariness of $f\in W^{1,\infty}(\Omega)$ allows us to conclude that $m_1=m_2$.
		\end{proof}

		We note that, in the uniqueness proof, we  used the fact that $m_1,m_2\in\mathcal{P}(\Omega)$, to have $\into\lambda_f dm_i(x)=\lambda_f$, $i=1,2$. Without this prescription, we would lose the uniqueness because, if $m $ solves \eqref{eq:fp}, then $km $ solves the same equation for all $k\in\R$.

		\subsection{Regularity of $m$.}
		%\label{sec:4}
		Now we prove refined estimates on the boundary behavior of $m$, and then more regularity of the solution, under extra assumptions on the drift $b$ near the boundary. 
		
		We start by recalling  the following classical version  of the strong maximum principle, which includes Hopf boundary lemma  (see e.g. \cite{PW}).
		
		\begin{prop}\label{strongmp} Let $\Omega$ be a bounded set satisfying the interior sphere condition at the boundary.  Given $b, a_{ij}, c(x)$  bounded (and sufficiently smooth), and such that $A(x)=(a_{ij})$ is uniformly elliptic, let us define
			$$
			L(u):= -\sum_{ij} a_{ij} u_{x_ix_j} + b(x) \cdot \nabla u +c(x) u\,.
			$$
			If $u\in C^2(\Omega)\cap C^1(\overline \Omega)$ satisfies $L(u) \geq 0$ in $\Omega$, and $u\geq 0$ in $\overline \Omega$, then either $u>0$ in $\Omega$ or $u\equiv 0$. Moreover, if   $u(x_0)=0$ for some $x_0\in \partial \Omega$, then $\frac{\partial u}{\partial \nu}(x_0)<0$.
		\end{prop}
		
		We can now show the following estimates on the boundary behavior of $m$.
		
		\begin{thm} Assume that the hypotheses of Theorem \ref{thm:exm} hold true, and in addition that $\exists\gamma>1,\delta_0,\theta>0$ such that $\forall\, x\in\Gamma_{\delta_0}$ it holds
			\begin{equation}\label{hp:mraff}
				\begin{split}
					b(x)=\frac{\gamma}{d(x)}\Big[\nu(x)+O\big(d(x)^\theta\big)\Big]\,,\qquad\mathrm{div}(b(x))=\frac{\gamma}{d(x)^2}\Big[1+O\big(d(x)^\theta\big)\Big]\,.
				\end{split}
			\end{equation}
			Let $m$ be the unique weak solution of \rife{eq:fp}. Then $m\in C^{1,\alpha}(\overline\Omega)$, for some $\alpha>0$,  and there exist $C_1,C_2>0$ such that
			\begin{equation}\label{eq:regm}
				C_1\,d(x)^{\gamma}\le m(x)\le C_2\, d(x)^{\gamma}\,.
			\end{equation}
		\end{thm}
		\begin{proof}
			We consider, for $\eps>0$, the solution $m_\eps$ of problem \rife{eq:m_epsapprox}. Since $b$ is smooth in $\Omega_\eps$, we can apply Proposition \ref{strongmp} to infer that $m_\eps>0$ in $\Omega_\eps$; moreover, since $\frac{\partial m_\eps}{\partial \nu}= -m_\eps b\cdot \nu $, by Hopf lemma (i.e. the boundary condition of Proposition \ref{strongmp}) we have that $m_\eps$ cannot vanish at the boundary. Hence, we have $m_\eps >0$ in $\overline \Omega_\eps$.  Now we consider the function $\phi= d(x)^{\gamma}-d(x)^{\gamma+\sigma}$ in the domain $\Omega\setminus \Omega_{\de_0}$, for $\de_0$ and $\sigma$ sufficiently small. Taking into account \eqref{hp:mraff}, a straightforward computation gives, for $\sigma<\theta$,
			$$
			-\Delta\phi- b(x)\cdot\nabla\phi- \phi\, \dive(b(x))=  d(x)^{ \gamma+\sigma-2}\Big((\gamma-1+\sigma)\sigma+o(1)\Big)\,,
			$$
			hence $\phi$ is a supersolution in $\Omega\setminus \Omega_{\de_0}$, for $\de_0$ sufficiently small. We define
			$$
			\Lambda:= \inf\{\lambda >0\,:\, \lambda \phi > m_\eps \quad \hbox{in $\Omega_\eps\setminus \Omega_{\de_0}$}\}.
			$$
			We clearly have $\Lambda <\infty$, because $m_\eps$ is bounded, and $\Lambda>0$ since $m_\eps$ is positive. Moreover, by continuity and definition of $\Lambda$,  there exists some point $x_0\in \overline{\Omega_\eps\setminus \Omega_{\de_0}}$ such that $m_\eps(x_0)= \Lambda \phi(x_0)$.
			Notice that, if $L(m):= -\Delta m- b\cdot \nabla m-m\,\dive(b)$, we have
			$$
			\Lambda \phi -m_\eps\geq 0\,,\qquad L(\Lambda \phi -m_\eps) >0 \quad \hbox{in $\Omega_\eps\setminus \Omega_{\de_0}$} 
			$$
			so $x_0 \not\in \Omega_\eps\setminus \Omega_{\de_0}$, because $\Lambda \phi -m_\eps>0$ in the interior due to Proposition \ref{strongmp}. In addition, at $\partial \Omega_\eps$ we have
			\begin{align*}
				\frac{\partial(\Lambda \phi -m_\eps)}{\partial \nu} & =   m_\eps b\cdot \nu - \Lambda \big(\gamma d(x)^{\gamma-1}-(\gamma+\sigma)d(x)^{\gamma+\sigma-1}\big)
				\\
				& =   (m_\eps- \Lambda \phi)  b\cdot \nu + \Lambda d(x)^{\gamma+\sigma-1} (\sigma+ o(1)) \,
			\end{align*}
			where we used \rife{hp:mraff} and $\theta>\sigma$. Hence,  $m_\eps(x_0)= \Lambda \phi(x_0)$ would imply  $\frac{\partial(\Lambda \phi -m_\eps)}{\partial \nu}>0$ at $x_0$. This excludes that $x_0\in \partial \Omega_\eps$. We deduce that $x_0\in \{d(x)= \delta_0\}$, which implies 
			$$
			\Lambda =  \de_0^{-\gamma}(1-\de_0^\sigma)^{-1}m_\eps(x_0) \leq   \de_0^{-\gamma}(1-\de_0^\sigma)^{-1} \sup_{\{d(x)= \delta_0\}} m_\eps\,.
			$$
			We proved so far that 
			$$
			m_\eps (x) \leq  \frac{d(x)^{\gamma}-d(x)^{\gamma+\sigma}}{\de_0^{\gamma}(1-\de_0^\sigma)} \sup_{\{d(x)= \delta_0\}} m_\eps\le C_{\delta_0}d(x)^\gamma\sup_{\{d(x)= \delta_0\}} m_\eps\qquad \forall x\in \Omega_\eps\setminus \Omega_{\de_0}\,.
			$$ 
			Recalling that $m_\eps$ has unit mass, and that $b$ is smooth inside, by elliptic regularity there exists a constant $k_0$ such that 
			$$
			m_\eps(x) \leq k_0  \int_{\Omega_\eps} m_\eps(y)\,dy = k_0 \qquad \forall x\,:\, d(x) >\frac{\de_0}2\,.
			$$
			We conclude that, for some constant $C_2$, independent of $\eps$, it holds
			\begin{equation}\label{eq:miserveperunicita}
				m_\eps(x) \leq C_2 \, d(x)^{\gamma} \,.
			\end{equation}
			Similarly, we prove the estimate from below. In this case we use $\phi= d(x)^{\gamma}+d(x)^{\gamma+\sigma}$ and we observe that $\phi$ is a subsolution in a suitable neighborhood of the boundary. Then we define
			$$
			\Theta:= \sup\{\lambda >0\,:\, \lambda \phi < m_\eps \quad \hbox{in $\Omega_\eps\setminus \Omega_{\de_0}$}\}
			$$ 
			and we observe that $ \Theta>0$ because $m_\eps$ is strictly positive in $\Omega_\eps$. As before, we use that 
			$$
			m_\eps- \Theta \phi\geq 0 \,,\qquad L(m_\eps- \Theta\phi)>0 \quad\hbox{in $\Omega_\eps\setminus \Omega_{\de_0}$}\,, 
			$$
			and
			$$
			\frac{\partial( m_\eps- \Theta \phi)}{\partial \nu} = - (m_\eps- \Theta \phi) b(x) \cdot \nu + \Theta d(x)^{\gamma+\sigma-1} (\sigma+ o(1))\quad \hbox{at $\partial \Omega_\eps$}\,.
			$$
			Using the strong maximum principle of Proposition \ref{strongmp}, we deduce once again that $m_\eps$ and $ \Theta \phi$ can only touch at the interior boundary, whence $\Theta \geq \de_0^{-\gamma}(1-\de_0^{\sigma})^{-1} \inf\limits_{\{d(x)= \de_0\}} m_\eps$. We deduce the lower estimate
			$$
			m_\eps (x) \geq  C_{\de_0}d(x)^{\gamma} \inf_{\{d(x)= \de_0\}} m_\eps\qquad \forall x\in \Omega_\eps\setminus \Omega_{\de_0}\,.
			$$ 
			Using Harnack inequality, and the fact that  $\{d(x)= \de_0\}$ is a compact set, we can estimate $m_\eps$ from below; this means that there exists $\tilde k_0$ such that
			$$
			\inf_{\{d(x)= \de_0\}} m_\eps \geq \tilde k_0 \int_{\Omega_{\frac{\de_0}2}}m_\eps(y)\,dy\,.
			$$
			The last term is estimated from below as in the proof of Theorem \ref{thm:exm}, using \rife{mass-eta}. We conclude that there exists a constant $C_1>0$, independent of $\eps$, such that
			$$
			m_\eps(x) \geq C_1 \, d(x)^{\gamma} \,.
			$$
			This concludes the proof of \eqref{eq:regm}, and consequently we have $m\in L^\infty(\Omega)$. We now observe that $bm_\eps\in L^\infty(\Omega)$; indeed,  
			$b$ and $m_\eps$ are bounded for $x\in\Omega_{\delta_0}$, and for $x\in\Gamma_{\delta_0}$ we have, due to \rife{hp:mraff} and \rife{eq:regm},
			$$
			|b|m_\eps\le Cd(x)^{\gamma -1}(1+o(1))\le C\,.
			%\qquad\hbox{for $\delta$ sufficiently small}\,.
			$$
			Therefore $m_\eps $   satisfies $-\Delta m_\eps= \dive(f_\eps)$ for some vector field which is uniformly bounded. By elliptic regularity, not only we have $m_\eps$ bounded in  $H^1(\Omega)\cap L^\infty(\Omega)$ uniformly with respect to $\eps$, but also $m_\eps$ bounded in $C^{1,\alpha}(\Omega_\eps)$ for some $\alpha\in (0,1)$ (see e.g.  \cite{lieb88}). Passing to the limit we get the same regularity for $u$.
			%For the $H^1(\Omega)$ bound, we multiply \eqref{eq:m_epsapprox} by $m_\eps$ and integrate by parts, obtaining
			%\begin{equation}\label{eq:weakme}
			%\begin{split}
			%\int_{\Omega_\eps} |\nabla m_\eps|^2\,dx+\int_{\Omega_\eps} m_\eps\, b\cdot \nabla m_\eps\,dx=0
			%\end{split}
			%\end{equation}
			%With Young's inequality we get
			%$$
			%\left|\int_{\Omega_\eps} m_\eps\,b\cdot \nabla m_\eps\,dx\right|\le \miezz\int_{\Omega_\eps} |\nabla m_\eps|^2\,dx + \miezz\int_{\Omega_\eps} |b|^2\,m_\eps^2\,dx\le \miezz\int_{\Omega_\eps} |\nabla m_\eps|^2\,dx+C\,,
			%$$
			%since $b$ and $m_\eps$ are bounded for $x\in\Omega_{\delta_0}$, and for $x\in\Gamma_{\delta_0}$ we have
			%$$
			%|b|m_\eps\le Cd(x)^{\gamma-\delta-1}(1+o(1))\le C\qquad\hbox{for $\delta$ sufficiently small}\,.
			%$$
			%Plugging this estimate into \eqref{eq:weakme} we get $m_\eps\in H^1(\Omega_\eps)$ uniformly in $\eps$, which implies $m\in H^1(\Omega)$.
		\end{proof}

		\begin{rem}
			We observe that hypothesis \eqref{hp:mraff} is satisfied in the case $b=p|\nabla u|^{p-2}\nabla u$, where $u$ solves \eqref{eq:u} with $f\in C^{\alpha}(\Omega)$. Actually, the estimate on $b(\cdot)=-\alpha(\cdot)$ is proved by \eqref{stime-feed} with  $\gamma=p'\ge2$. Moreover, called $\{e_i\}_i$ the canonical base on $\R^n$, we get by \eqref{eq:Jacc}, and for $p\neq 2$,
			\begin{equation*}
				\begin{split}
					\mathrm{div}(b(x))=\sum\limits_{i=1}^n\langle (\mathrm{Jac}\,b)\,e_i,e_i\rangle=p'd(x)^{-2}\Big[1+O\big(\omega_d\big)\Big] =\frac{p'}{d(x)^2}\Big(1+O\big(d^\theta\big)\Big),
				\end{split}
			\end{equation*}
			where we can choose $\theta=1$ for $1<p<\frac32$, $\theta<1$ for $p=\frac32$, $\theta=p'-2$ for $\frac32<p<2$.
			
			For the case $p=2$ we can directly use the equation satisfied by $u$ and \eqref{eq:stimeDu}. Actually, for $p=2$ we have
			$$
			\mathrm{div}(b(x))=2\,\Delta u=2\,|\nabla u|^2+2\,\bl-2f(x)=\frac 2{d(x)^2}\big(1+O(d)\big)\,.
			$$
		\end{rem}
		\section{The Mean Field Games system}
		The well-posedness and the regularity of the Fokker-Planck equation, together with the literature on the Hamilton-Jacobi equation, allow us to prove existence and uniqueness of the Mean Field Games system. We start with the following existence result.
		
		\begin{thm}
			Assume that $1<p\le 2$ and $F\in C(\Omega \times \cP(\Omega))$ satisfies \rife{Fxm}, where $\Omega$ is a bounded domain with $\mathcal{C}^2$ boundary. Then the problem \eqref{mfg} admits a solution $(\bl,u,m)$, with $u\in W^{2,r}_{loc}(\Omega)$ for all $r<+\infty$ and $m\in L^\infty(\Omega)\cap H^1(\Omega)$. In addition, $m\in C^{1,\alpha}(\overline \Omega)$ and $(u,m)$ satisfy the estimates \rife{eq:stimeu}-\rife{eq:stimeDu} and \rife{eq:regm} with $\gamma=p'$.
		\end{thm}
		\begin{proof}
			We use   Schauder's fixed point Theorem. Let $X$ be the following space
			$$
			X=\left\{\mu\in H^1(\Omega)\,\,\big|\,\,\norm{\mu}_{H^1(\Omega)}\le M\,,\,\mu\ge 0\,,\,\into\mu(x)\,dx=1\right\}\,,
			$$
			for $M\ge0$ which will be chosen later. It is clear that $X$ is a compact convex subset in $L^2(\Omega)$. For $\mu\in X$ and $x_0\in\Omega$, we take $(\bl_\mu,u_\mu)$ as the unique solution of
			\begin{equation}\label{eq:umu}
				\begin{cases}
					-\Delta u_\mu+|\nabla u_\mu|^p+\bl_\mu=F(x,\mu)\,,\\
					\lim\limits_{x\to\partial\Omega}u_\mu(x)=+\infty\,,\qquad u_\mu(x_0)=0\,.
				\end{cases}
			\end{equation}
			We know that $u_\mu\in W^{2,r}_{loc}(\Omega)$ for all $r>1$ and   estimates \eqref{eq:stimeu}, \eqref{eq:stimeDu}, \eqref{eq:stimeD2u},  \eqref{eq:taylorD2u} hold uniformly in $\mu$, due to assumption \rife{Fxm}.
			
			Then, we define $\Phi(\mu)=m_\mu$, where $m_\mu$ solves
			\begin{equation}\label{eq:mmu}
				\begin{cases}
					\displaystyle -\Delta m_\mu-p\,\mathrm{div}(m_\mu|\nabla u_\mu|^{p-2}\nabla u_\mu)=0\,,\\
					\displaystyle m_\mu\ge 0\,,\qquad\into m_\mu(x)\,dx=1
				\end{cases}
			\end{equation}
			in the sense of Definition \ref{def:fp}. From the previous results we know that $m_\mu\in L^\infty(\Omega)\cap H^1(\Omega)$, and estimate \eqref{eq:regm} holds.
			
			The uniform $H^1$ bound for $m_\mu$ implies that, for $M$ sufficiently large, $\Phi(X)\subseteq X$. We only have to prove the continuity of the map $\Phi$.
			
			To do that, let $\mu_k\to\mu$ in $L^2$, with $\mu_k,\mu\in X$. For simplicity, we use the notation $u_k$, $m_k$, $\bl_k$ instead of $u_{\mu_k}$, $m_{\mu_k}$, $\bl_{\mu_k}$. Moreover, we will shortly write $b_k$ to denote the vector fields $p|\nabla u_k|^{p-2}\nabla u_k$.
			
			Since $\bl_k$ is bounded, and $u_k$ is bounded  in $W^{2,2} (\cK)$, for every compact subset $\cK\subset \Omega$,   up to a (non-relabeled) subsequence we have $\bl_k \to \bl $ and $u_k\to u$   strongly in $H^1 (\cK) $ and  almost everywhere. Passing to the limit in the weak formulation of $(\bl_k,u_k)$ we have that $(\bl,u)$ solves \eqref{eq:umu}.
			
			Since $m_k$ is bounded in $H^1(\Omega)\cap L^\infty(\Omega)$, we have up to subsequences $m_k\to m$ a.e., strongly in $L^p$ $\forall p\ge1$ and weakly in $H^1(\Omega)$, for a certain $m\in L^\infty\cap H^1(\Omega)$. In order to obtain $\Phi(\mu)=m$ and conclude the proof, we have to prove that $m$ satisfies \eqref{eq:mmu} in the sense of Definition \ref{def:fp}.
			
			To do that, we take $\phi\in L^\infty$ satisfying $-\Delta\phi+p|\nabla u|^{p-2}\nabla u\cdot\nabla\phi\in W^{1,\infty}$. For simplicity, we call $b:=p|\nabla u|^{p-2}\nabla u$ and $\eta:=-\Delta\phi+b(x)\cdot\nabla\phi$.
			
			Let $\phi_{k}\in L^\infty$ be the solution of
			$$
			-\Delta\phi_{k}+b_k(x)\cdot\nabla\phi_{k}=\eta+\lambda_k\in W^{1,\infty}(\Omega)\,,
			$$
			for a certain $\lambda_k\in\R$. Such a solution exists thanks to Proposition \ref{prop:delta0}. The weak formulation of $m_k$ implies
			\begin{equation*}
				\begin{split}
					\into(\eta+\lambda_k)m_k(x)\,dx=\into\Big(-\Delta\phi_k(x)+b_k(x)\cdot\nabla\phi_k(x))\Big)\,dm_k(x)=0
				\end{split}
			\end{equation*}
			From \eqref{eq:boundlambda} we have $|\lambda_k|\le\norminf{\eta}\le C$, then $\exists\lambda$ such that $\eta+\lambda_k\to -\Delta\phi+b(x)\cdot\nabla\phi+\lambda$ in $L^2(\Omega)$. Since $m_k\to m$ in $L^2(\Omega)$, we can pass to the limit and obtain
			\begin{equation*}
				\begin{split}
					\into \Big(-\Delta\phi+b(x)\cdot\nabla\phi+\lambda\Big) dm(x)=0\,.
				\end{split}
			\end{equation*}
			We only have to prove that $\lambda=0$ to show that \eqref{eq:fpformula} holds and conclude the proof. The weak formulation of the equation of $\phi_k$ implies that
			\begin{equation*}
				\begin{split}
					\into\nabla\phi_k\cdot(\nabla\xi+b_k\,\xi)\,dx=\into(\eta+\lambda_k)\,\xi\,dx\,,
				\end{split}
			\end{equation*}
			for all $\xi\in C_c^\infty(\Omega)$. 
			\begin{comment}
				By density, this equation holds for all $\xi\in H_0^1(\Omega)$. Take $\xi=\phi_k\zeta^2$, with $\zeta\inC_c^\infty(\Omega)$ and $supp(\zeta)\subset V$, where $V$ is a compact subset of $\Omega$. We get
				\begin{equation*}
					\begin{split}
						\into|\nabla\phi_k|^2\zeta^2\,dx\le C\int_V|\phi_k|\zeta^2\,dx-\int_V 2\zeta\phi_k\nabla\zeta\cdot\nabla\phi_k-\int_V \phi_k\zeta^2\nabla\phi_k\cdot b_k\,.
					\end{split}
				\end{equation*}
			\end{comment}
			We know that $\phi_k$ is bounded in $W^{1,r}(\Omega)$ $\forall\,r\ge1$ (including $r=+\infty$ in the case $p\neq2$) and the bound depends on $\norm{\eta+\lambda_k}_{W^{1,\infty}}$, which does not depend on $k$. Then there exists $\tilde\phi\in H^1(\Omega)$ such that $\nabla\phi_k\weak\nabla\tilde\phi$ in $L^2(\Omega)$. Moreover, $b_k$ is locally uniformly bounded and $b_k\to b$ almost everywhere, which implies $b_k\to b$ in $L^2_{loc}(\Omega)$. Passing to the limit we find
			$$
			\into\nabla\tilde\phi\cdot(\nabla\xi+b\,\xi)\,dx=\into(\eta+\lambda)\,\xi\,dx\,.
			$$
			This means that $\phi$ and $\tilde\phi$ solve  the two equations
			$$
			-\Delta\phi+b(x)\cdot\nabla\phi=\eta\,,\qquad-\Delta\tilde\phi+b(x)\cdot\nabla\tilde\phi=\eta+\lambda\,.
			$$
			Thanks to Proposition \ref{prop:delta0}, this implies $\lambda=0$ and concludes the proof.
		\end{proof}
		
		For the uniqueness part we need the following Lemma, which allows us to extend the set of test functions for the Fokker-Planck equation.
		\begin{lem}\label{ucontrom}
			Assume $m$ is a solution of \eqref{eq:fp}, with $b\in W^{1,\infty}_{loc}(\Omega)$ satisfying \eqref{hp:mraff} with $\gamma=p'$. Then, if $u$ solves \eqref{eq:u} with $f\in C^\alpha(\Omega)$, it holds
			\begin{equation}\label{eq:utestfunction}
				\into (-\Delta u+b(x)\cdot\nabla u)\,dm(x)=0\,.
			\end{equation}
		\end{lem}
		\begin{proof}
			Let $m_\eps$ solve \eqref{eq:m_epsapprox}. Since $u$ is smooth in $\Omega_\eps$, we can use $u$ as test function for $m_\eps$, obtaining
			$$
			\int_{\Omega_\eps} (-\Delta u+b(x)\cdot\nabla u)\,m_\eps(x)\,dx=-\int_{\partial\Omega_\eps}m_\eps\,\nabla u\cdot\nu\,dx\,.
			$$
			For the right-hand side, using \eqref{eq:stimeDu} and \rife{eq:miserveperunicita} (with $\gamma=p'$), we have  
			$$
			\left|\int_{\partial\Omega_\eps}m_\eps\,\nabla u\cdot\nu\,dx\right|\le C\eps^{p' }\eps^{-\frac 1{p-1}}=C\eps \to0\,.
			$$
			For the left-hand side, we use \eqref{eq:stimeDu}, \eqref{eq:stimeD2u} and \eqref{eq:miserveperunicita} (with $\gamma=p'$), so we have 
			$$
			|(-\Delta u+b(x)\cdot\nabla u)\,m_\eps(x)|\le   C \,.
			$$
			Since $m_\eps\to m$ a.e., by  dominated convergence theorem we  get \eqref{eq:utestfunction}.
		\end{proof}
		Now we are ready to prove the uniqueness part.
		\begin{thm}
			Assume that $1<p\le 2$ and $F$ satisfies \rife{Fxm} and \rife{LL-mon}.  Then the problem \eqref{mfg} admits at most one solution $(\bl,u,m)$, up to an additive constant for the function $u$.
		\end{thm}
		\begin{proof}
			Let $(\bl,u,m)$ and $(\tilde\lambda,v,\mu)$ be two solutions. Then it holds
			\be\label{u-v}
			-\Delta(u-v)+|\nabla u|^p-|\nabla v|^p+\bl-\tilde\lambda=F(x,m)-F(x,\mu)\,.
			\ee
			%\begin{equation*}
			%\begin{split}
			%&-\Delta(u-v)+|\nabla u|^p-|\nabla v|^p+\bl-\tilde\lambda=F(x,m)-F(x,\mu)\,,\\
			%&-\Delta(m-\mu)-p\,\mathrm{div}(m|\nabla u|^{p-2}\nabla u-\mu|\nabla v|^{p-2}\nabla v)=0\,.
			%\end{split}
			%\end{equation*}
			By Lemma \ref{ucontrom}, both $u$ and $v$ satisfy \rife{eq:utestfunction} with both $m$  and $\mu$. This implies
			\begin{equation*}
				\begin{split}
					& \into -\Delta(u-v)\, dm(x)+ p \into |\nabla u|^{p-2}\nabla u\nabla(u-v)\, dm(x) =0 
					\\ & \into -\Delta(u-v)\, d\mu(x)+ p \into |\nabla v|^{p-2}\nabla v\nabla(u-v)\, d\mu(x) =0\,.
				\end{split}
			\end{equation*}
			Using \rife{u-v} in  the above equalities, and subtracting one from the other,  we obtain
			\begin{equation*}
				\begin{split}
					&\into \big(F(x,m)-F(x,\mu)\big)(m-\mu)\,dx \\
					&+ \into\Big(|\nabla v|^p-|\nabla u|^p-p|\nabla u|^{p-2}\nabla u(\nabla v-\nabla u)\Big)\,dm(x)\\
					&+ \into \Big(|\nabla u|^p-|\nabla v|^p-p|\nabla v|^{p-2}\nabla v(\nabla u-\nabla v)\Big)\,d\mu(x)= 0\,.
				\end{split}
			\end{equation*}
			The first term is non-negative, due to \rife{LL-mon}, whereas  the other two terms are  non-negative due to the strict convexity of $x\to|x|^p$ for $p>1$.
			Therefore, we deduce that  $\nabla u=\nabla v$ almost everywhere on the set $\{m\neq0\}\cup\{\mu\neq 0\}$. Hence, $m$ and $\mu$ solve the same Fokker-Planck equation, which implies $m=\mu$ by Theorem \ref{thm:exm}. This means that $(\bl,u)$ and $(\tilde\lambda,v)$ solve the same Hamilton-Jacobi equation.  Hence, by Proposition  \ref{prop:delta0}, we get  $\bl=\tilde\lambda$, $u=v+C$ for a certain $C>0$, and the proof is completed.
		\end{proof}
		
		{\bf Acknowledgements.} The research was supported by Project ``Mean-field games: models, theory, and computational aspects" ORA-2021-CRG10-4674.5
		(Kaust University).  A.P. is supported by Indam (Istituto Nazionale di Alta Matematica) and GNAMPA  research projects.


\begin{thebibliography}{abc99xyz} 
			
			\bibitem{BCR} Bardi, M., Cesaroni, A., Rossi, L. (2016). {\it Nonexistence of nonconstant solutions of some degenerate Bellman equations and applications to stochastic control}. ESAIM: Control, Optimisation and Calculus of Variations, 22(3), 842-861.
			
			\bibitem{BPT} Barles, G., Porretta, A. Tabet Tchamba, T. (2010). {On the Large Time Behavior of Solutions of the Dirichlet problem for Subquadratic Viscous Hamilton-Jacobi Equations}, Journal. Math. Pures et Appl. 94,  497-519.
			
			
			\bibitem{Krylov} Bogachev, V.I., Krylov, N.V., R\"ockner, M., Shaposhnikov, S.V. (2015). {\it Fokker-Planck-Kolmogorov Equations}. Mathematical Surveys and Monographs, 207.
			
			\bibitem{CaCa}  Cannarsa, P. and  Capuani R. (2018). {\it Existence and uniqueness for mean field games with state constraints}. PDE models for multi-agent phenomena, 49--71, Springer INdam Ser. 28. 
			
			\bibitem{CaCaCa}  Cannarsa P.,   Capuani R. and Cardaliaguet P. (2021). {\it Mean field games with state constraints: from mild to pointwise solutions of the PDE system}.  Calc. Var. and PDE, 60 (3), 1--32.
			
			\bibitem{CP-Cime} Cardaliaguet, P. and Porretta, A. (2020). {\sl An introduction to mean field game theory}, Springer-Cime Lecture Notes in Mathematics, Vol.
			2281 (2020), pp 1-158.
			
			\bibitem{CD} Carmona, R. and Delarue, F. (2018) {Probabilistic theory of mean field games with applications. I \& II}. Probability Theory and Stochastic Modelling, 83 \& 84. Springer, Cham, 2018. xxv+713  pp. \&  xxiv+697 pp.
			
			\bibitem{DistFunct} Delfour, M.C., Zolesio, J.-P. (1994). {\it Shape analysis via oriented distance function}. J. Funct. Anal., 123, 129-201.
			
			\bibitem{Droniou} Droniou, J., Vazquez, J.-L. (2009). {\it Noncoercive convection-diffusion elliptic problems with Neumann boundary conditions}. Calculus of Variations and Partial Differential Equations, Springer Verlag, 34 (4), 413-434.
			
			\bibitem{GT} Gilbarg, D., Trudinger, N.S. (2001). {\it Elliptic Partial Differential Equations of second-order}. Springer, Classics in Mathematics, vol. 224.
			
			\bibitem{HCM} Huang, M.,  Caines, P.E.,  Malham\'e,  R.P. (2006). {\it Large population stochastic dynamic games: closed-loop McKean-Vlasov systems and the Nash certainty equivalence principle},  Comm. Inf. Syst., 6, 221--251.
			
			\bibitem{LL1} Lasry, J.-M., Lions, P.-L. (2006). {\it Jeux \`a champ moyen. I. Le cas stationnaire.} C. R. Math. Acad. Sci. Paris  343,   619--625.
			
			\bibitem{LL2} Lasry, J.-M., Lions, P.-L. (2006). {\it Jeux \`a champ moyen. II. Horizon fini et contr\`ole optimal.} C. R. Math. Acad. Sci. Paris  343,   679--684.
			
			\bibitem{LL-japan} Lasry, J.-M., Lions, P.-L. (2007).  {\it Mean field games.}  Jpn. J. Math.  2,  no. 1, 229--260.
			
			\bibitem{LL-SC} Lasry, J.-M., Lions, P.-L. (1989). {\it Nonlinear elliptic equations with singular boundary conditions and stochastic control with state constraints. I. The model problem}. Math. Ann., 283, 583-630.
			
			\bibitem{LP} Leonori, T., Porretta, A., (2007). {\it The boundary behavior of blow-up solutions related to a stochastic control problem with state constraint}, { Siam J. Math. Anal.}, 39(4), 1295-1327.
			
			\bibitem{Por_Leon} Leonori, T., Porretta, A., (2011). {\it Gradient bounds for elliptic problems singular at the boundary}. Arch. Ration. Mech. Anal., 202(2),663-705.
			
			\bibitem{lieb88} G. Lieberman, {\it Boundary regularity for solutions of degenerate elliptic equations}, Nonlinear
			Anal. 12 (1988), 1203-1219.
			
			\bibitem{MFGinv} Porretta, A., Ricciardi, M., (2019). {\it Mean field games under invariance conditions for the state space}. Communications in Partial Differential Equations, 45(2), 1-45.
			
			\bibitem{JuanSebastianVeron} Porretta, A., V\'eron, L., (2006). {\it Asymptotic behavior of the Gradient of Large Solutions to Some Nonlinear Elliptic Equations}. Advanced Nonlinear Studies, 6, 351-378.
			
			\bibitem{PW} Protter, M.H., Weinberger, H.F. (1967). {\sl Maximum Principles in Differential Equations}, Prentice-Hall: Englewood Cliffs, New Jork.
			
		\end{thebibliography}
	\end{document}